\documentclass[a4paper,12pt]{amsart}
\title{Stability conditions and extremal contractions}
\date{}
\author{Yukinobu Toda}

\dedicatory{Dedicated to Professor Yujiro Kawamata on 
the occasion of his 60-th birthday}

\usepackage{amscd}
\usepackage{amsmath}
\usepackage{amssymb}
\usepackage{amsthm}
\usepackage{float}
\usepackage[dvips]{graphicx}

\usepackage[all,ps,dvips]{xy}

\usepackage{array}
\usepackage{amscd}
\usepackage[all]{xy}

\DeclareFontFamily{U}{rsfs}{%
\skewchar\font127}
\DeclareFontShape{U}{rsfs}{m}{n}{%
<-6>rsfs5<6-8.5>rsfs7<8.5->rsfs10}{}
\DeclareSymbolFont{rsfs}{U}{rsfs}{m}{n}
\DeclareSymbolFontAlphabet
{\mathrsfs}{rsfs}
\DeclareRobustCommand*\rsfs{%
\@fontswitch\relax\mathrsfs}

\theoremstyle{plain}
\newtheorem{thm}{Theorem}[section]
\newtheorem{prop}[thm]{Proposition}
\newtheorem{lem}[thm]{Lemma}

\newtheorem{defi}[thm]{Definition}
\newtheorem{rmk}[thm]{Remark}
\newtheorem{cor}[thm]{Corollary}

\newtheorem{prop-defi}[thm]{Proposition-Definition}
\newtheorem{thm-defi}[thm]{Theorem-Definition}
\newtheorem{lem-defi}[thm]{Lemma-Definition}

\newtheorem{question}[thm]{Question}

\newtheorem{conj}[thm]{Conjecture}

\newdimen\argwidth
\def\db[#1\db]{
 \setbox0=\hbox{$#1$}\argwidth=\wd0
 \setbox0=\hbox{$\left[\box0\right]$}
  \advance\argwidth by -\wd0
 \left[\kern.3\argwidth\box0 \kern.3\argwidth\right]}

\newcommand{\aA}{\mathcal{A}}
\newcommand{\bB}{\mathcal{B}}
\newcommand{\cC}{\mathcal{C}}
\newcommand{\dD}{\mathcal{D}}
\newcommand{\eE}{\mathcal{E}}
\newcommand{\fF}{\mathcal{F}}
\newcommand{\gG}{\mathcal{G}}
\newcommand{\hH}{\mathcal{H}}

\newcommand{\mM}{\mathcal{M}}

\newcommand{\oO}{\mathcal{O}}
\newcommand{\pP}{\mathcal{P}}
\newcommand{\qQ}{\mathcal{Q}}

\newcommand{\sS}{\mathcal{S}}
\newcommand{\tT}{\mathcal{T}}
\newcommand{\uU}{\mathcal{U}}
\newcommand{\vV}{\mathcal{V}}

\newcommand{\lr}{\longrightarrow}

\newcommand{\Supp}{\mathop{\rm Supp}\nolimits}
\newcommand{\Hom}{\mathop{\rm Hom}\nolimits}

\newcommand{\dotimes}{\stackrel{\textbf{L}}{\otimes}}
\newcommand{\dR}{\mathbf{R}}
\newcommand{\dL}{\mathbf{L}}

\newcommand{\id}{\textrm{id}}

\newcommand{\ch}{\mathop{\rm ch}\nolimits}

\newcommand{\Ext}{\mathop{\rm Ext}\nolimits}
\newcommand{\Spec}{\mathop{\rm Spec}\nolimits}
\newcommand{\rank}{\mathop{\rm rank}\nolimits}
\newcommand{\Coh}{\mathop{\rm Coh}\nolimits}

\newcommand{\cneq}{\mathrel{\raise.095ex\hbox{:}\mkern-4.2mu=}}
\newcommand{\eqcn}{\mathrel{=\mkern-4.5mu\raise.095ex\hbox{:}}}

\newcommand{\Cok}{\mathop{\rm Cok}\nolimits}

\newcommand{\Stab}{\mathop{\rm Stab}\nolimits}
\newcommand{\PPer}{\mathop{\rm Per}\nolimits}
\newcommand{\oPPer}{\mathop{\rm ^{0}Per}\nolimits}

\newcommand{\iPPer}{\mathop{\rm ^{-1}Per}\nolimits}

\newcommand{\ppPPer}{\mathop{^{{p}}\rm{Per}}\nolimits}

\newcommand{\modu}{\mathop{\rm mod}\nolimits}
\newcommand{\End}{\mathop{\rm End}\nolimits}

\newcommand{\Imm}{\mathop{\rm Im}\nolimits}

\newcommand{\Ker}{\mathop{\rm Ker}\nolimits}
\newcommand{\Ree}{\mathop{\rm Re}\nolimits}

\begin{document}
\begin{abstract}
We show that any extremal contraction
from a smooth projective variety with dimension 
less than or equal to three appears 
as a moduli space of 
(semi)stable objects in the derived category
of coherent sheaves. 
\end{abstract}

\maketitle

\setcounter{tocdepth}{1}
\tableofcontents

\section{Introduction}
\subsection{Motivation}
Let $X$ be a smooth projective variety over 
$\mathbb{C}$. 
Recall that a Minimal Model Program (MMP)
is a sequence of
divisorial contractions or flips
\begin{align*}
X=X_0 \dashrightarrow X_1 \dashrightarrow X_2 \dashrightarrow
 \cdots \dashrightarrow X_{N}=X_{\rm{min}}
\end{align*}
such that $X_{\rm{min}}$ is either a minimal model
 (i.e. $K_{X_{\rm{min}}}$ is nef)
or has a Mori fiber space structure. 
In two dimensional case, the MMP is just contracting $(-1)$-curves. 
In three dimensional case, the MMP is completed in 1980's 
by allowing some mild singularities on each $X_i$. 
(cf.~\cite{KM}.)
In a higher dimensional case, the MMP is known to exist 
when Kodaira dimension of $X$ is equal to
the dimension of $X$.
(cf.~\cite{BCHM}.)

On the other hand, a surprising relationship between 
MMP and derived categories of coherent sheaves was 
found by Bondal and Orlov~\cite{B-O2}. 
They observed a phenomena that the derived category 
gets smaller by MMP, at least each $X_i$ is smooth 
and birational map $X_i \dashrightarrow X_{i+1}$ is a 
standard one. 
This result was generalized
by Bridgeland~\cite{Br1} for arbitrary 
three dimensional flops, and by Kawamata~\cite{Kawlog}
for toroidal cases. 
It is now an interesting research subject to 
see an interaction between MMP and derived category.

In this paper, we study
a relationship between MMP and derived category 
from the viewpoint of stability conditions
and moduli spaces.  
A relationship between MMP and stability 
conditions was first pointed out in~\cite{Thad}, 
in which flips appeared
as a variation of GIT stability. This relationship
was also applied for the study of 
wall-crossing of moduli spaces of sheaves~\cite{MW}. 
The derived category appeared in this context in~\cite{Br1}, 
in which three dimensional flops were constructed 
as moduli spaces of objects in the derived category. 
This result should be interpreted 
that three dimensional 
flops are obtained
as variations of 
stability conditions in the derived category,  
and we expect that any birational 
map which appears in MMP is always realized 
in this way. 

Now there is a notion of 
stability conditions on derived categories by Bridgeland~\cite{Brs1}, 
which provides a 
mathematical formulation of Douglas's $\Pi$-stability~\cite{Dou2}. 
We address the following question in this paper: 
\begin{question}\label{question}
Is each $X_i$ a moduli space of Bridgeland (semi)stable objects
in the derived category of $X$, 
and MMP is interpreted as wall-crossing 
under a variation of Bridgeland stability conditions?
\end{question}
If the answer to the above question is true, 
then we are able to analyze
the geometry of each $X_i$, especially 
the minimal model $X_{\rm{min}}$, from
a categorical data of $D^b \Coh(X)$. 
The purpose of this paper is to answer the above 
question for the first step of MMP
when $\dim X \le 3$, that is an 
 \textit{extremal contraction}. 
This is a birational morphism
\begin{align}\label{extremal}
f \colon X \to Y
\end{align}
such that $Y$ is a normal projective variety, 
$-K_X$ is $f$-ample and the relative Picard 
number of $f$ is equal to one. 
When $\dim X=2$, then $f$ is just a contraction of a 
$(-1)$-curve. When $\dim X=3$, 
$f$ contracts a divisor $D$ to a curve or point in $Y$, 
and it is classified by Mori~\cite{Mori}.

\subsection{Result for surfaces}
We first study the case of $\dim X=2$. 
Let us fix the notation: 
for a Bridgeland stability condition
$\sigma$ on $D^b \Coh(X)$, 
we denote by 
 $M^{\sigma}([\oO_x])$
the set of isomorphism classes of 
$\sigma$-semistable objects $E$
with phase one,
satisfying $\ch(E)=\ch(\oO_x)$ for $x\in X$. 
The following is the result for the surface case.  
\begin{thm}\label{thm:intro1}
\emph{{\bf (Theorem~\ref{thm:main1})}}
Let $X$ be a smooth projective surface and $f\colon X \to Y$
an extremal contraction. Then there is a
one parameter family of Bridgeland stability 
conditions $\{\sigma_t\}_{t \in (-1, 1)}$
on $D^b \Coh(X)$ 
satisfying the following: 
\begin{itemize}
\item If $t<0$, then $X$ is the fine moduli space of 
$\sigma_t$-stable objects in $M^{\sigma_t}([\oO_x])$. 
\item If $t=0$, then $Y$ is the coarse moduli space of 
$S$-equivalence classes\footnote{The notion of `\textit{S-equivalence}'
is a direct analogue used in the study of 
moduli of sheaves. See Definition~\ref{S-eq}.} 
 of objects in $M^{\sigma_t}([\oO_x])$. 
\item If $t>0$, then $Y$ is the fine moduli space of $\sigma_t$-stable
objects in $M^{\sigma_t}([\oO_x])$. 
\end{itemize}
 \end{thm} 
The above result is based on the following 
expectation on the 
relationship between the space of 
stability conditions and the nef cone of $X$: 
let us consider the subset of 
the space of Bridgeland stability conditions 
$(Z, \aA)$ so that all the skyscraper sheaves $\oO_x$
are stable of phase one. (We call it a \textit{geometric chamber}.)
Then $\Imm Z$ restricted to 
the curve classes determines an ample class on 
$X$, by the argument of~\cite[Lemma~10.1]{Brs2}.
So the geometric chamber is closely related to the ample cone
of $X$. 
Now if we consider the boundary of the 
nef cone given by the pull-back of the ample cone 
of $Y$, we expect that there is a
corresponding boundary of the geometric chamber
in the space of stability conditions. 
If this is true, it would be interesting to investigate
the wall-crossing phenomena of the moduli spaces
of (semi)stable objects 
after crossing the above boundary (wall). 

Indeed, the stability conditions $\sigma_t$ for $t<0$
are contained in the geometric chamber 
constructed by Arcara and Bertram~\cite{AB}, 
and $\sigma_0$ lies at 
its boundary corresponding to the 
pull-back of the ample cone of $Y$
to $X$. 
The stability conditions $\sigma_t$ for $t>0$ are 
obtained by crossing the boundary at $\sigma_0$, 
although  
it is not clear how to describe 
them explicitly. 
Now Theorem~\ref{thm:intro0} 
answers Question~\ref{question} for the contraction $f\colon X \to Y$: 
it is 
realized by crossing 
the wall given by the boundary of the geometric chamber. 

More precisely, the
 one parameter family $\{\sigma_t\}_{t\in (-1, 1)}$ is constructed as 
follows:
we first construct $\sigma_0$ 
to be a pair
\begin{align*}
\sigma_{0}=(Z_{f^{\ast}\omega}, \bB_{f^{\ast}\omega})
\end{align*}
for an ample divisor $\omega$ on $Y$. 
The central charge $Z_{f^{\ast}\omega}$ is given by 
\begin{align}\label{central}
Z_{f^{\ast}\omega}(E)=-\int_{X} e^{-if^{\ast}\omega} \ch(E),
\end{align}
and the heart $\bB_{f^{\ast}\omega}$ is 
a tilting of the
perverse heart $\PPer(X/Y)$ in the sense of~\cite{Br1}.
The tilting here is similar to the one constructed
in~\cite{Brs2}, \cite{AB} applied for $\Coh(X)$.  
Then we deform $\sigma_0$
to a one parameter family $\{\sigma_{t}\}_{t\in (-1, 1)}$,
so that 
any object $\oO_x$ with $x\in X$ becomes 
$\sigma_t$-stable for $t<0$, and 
any 
object $\dL f^{\ast}\oO_y$ with $y\in Y$ 
becomes $\sigma_t$-stable for $t>0$. 

The most serious technical issue is to 
show that $\sigma_0$ satisfies the condition 
which guarantees the existence of the deformation, 
formulated as a \textit{support property}. 
The key ingredient to prove the support property 
is an analogue of Bogomolov-Gieseker (BG) 
inequality for certain semistable 
objects in the perverse heart $\PPer(X/Y)$.
By combining it with the techniques
developed in~\cite{Bay}, \cite{BMT}, 
we establish the BG inequality for 
$\sigma_0$-semistable objects in $\bB_{f^{\ast}\omega}$.
By using the above BG inequality, 
we are able to evaluate Chern characters of 
$\sigma_0$-semistable objects,
and prove the 
support property for $\sigma_0$. 

\subsection{Result for 3-folds I: Derived equivalence}
In the three dimensional case, 
the extremal contraction (\ref{extremal})
is classified into the following five types~\cite{Mori}: 

{\bf Type I:} $Y$ is smooth and $f$ is a blow up at a smooth 
curve. 

{\bf Type II:} $Y$ is smooth and $f$ is a blow-up 
at a point. 

{\bf Type III:} $Y$ has an ordinary double point, and 
$f$ is a blow-up at the singular point. 

{\bf Type IV:}
 $Y$ has an orbifold singularity $\mathbb{C}^3/(\mathbb{Z}/2\mathbb{Z})$
and $f$ is a blow-up at the singular point. 

{\bf Type V:} $Y$ has a $cA_2$-singularity and $f$ is a blow-up 
at the singular point.

We begin the study of three dimensional case 
with the construction of the perverse heart
and relating it to a sheaf 
of non-commutative algebras on $Y$. 
When $f(D)$ is a curve, 
Van den Bergh~\cite{MVB} shows that 
$X$ is derived equivalent to
a certain coherent $\oO_Y$-algebras
$\mathrsfs{A}$, and 
the perverse heart in~\cite{Br1} corresponds to 
the category of right $\mathrsfs{A}$-modules. 
When $f(D)$ is a point,
we show that 
the  
previous result~\cite{TU} 
on the local construction of the perverse 
heart and a derived equivalence 
with non-commutative algebras 
can be applied to our situation. 
By improving the argument of~\cite{TU}, we show the following:
\begin{thm}\label{thm:intro0}
\emph{{\bf (Theorem~\ref{thm:tilting})}}
Let $X$ be a smooth projective 3-fold and $f \colon X \to Y$
an extremal contraction. Then there is a vector bundle $\eE$
on $X$ and a derived equivalence
\begin{align}\label{derived:e}
\dR f_{\ast} \dR \hH om(\eE, \ast) \colon 
D^b \Coh(X) \stackrel{\sim}{\to} D^b \Coh(\mathrsfs{A}). 
\end{align}
Here $\mathrsfs{A} \cneq f_{\ast} \eE nd(\eE)$ and 
$\Coh(\mathrsfs{A})$ is the category of coherent right 
$\mathrsfs{A}$-modules on $Y$. 
\end{thm}
The construction of $\eE$ 
in Theorem~\ref{thm:intro0} is 
more concrete than the one in~\cite{TU}, 
and we can explicitly study it. 
We note that 
it is not difficult to construct
$\eE$
in types I and II: 
in type I, the result is contained in~\cite{MVB}.
In type II, $\eE$ is constructed by 
taking the direct sum of line bundles 
associated to the exceptional divisor $D\cong \mathbb{P}^2$, 
which restrict to a full 
exceptional collection on $D$.  
However it is not obvious to construct 
$\eE$ in other cases:
\begin{itemize}
\item In types III and IV, there is a full exceptional 
collection of line bundles
 on $D$, and their direct sum
extends to a
formal neighborhood of $D$.
However it
 does not necessary extend to the whole space $X$. 
\item In type V, 
$D$ is a singular quadric surface in $\mathbb{P}^3$, and 
there is no full exceptional collection
on it. 
In this case, the construction of $\eE$ is
new even in a formal neighborhood of $D$. 
\end{itemize}

Similarly to the surface case, 
there is a perverse heart $\PPer(X/Y)$
in $D^b \Coh(X)$ which 
corresponds to $\Coh(\mathrsfs{A})$ under the 
derived equivalence (\ref{derived:e}).
The information of simple objects in $\PPer(X/Y)$
is required for the study of stability conditions, 
and the above result enables us to describe 
them explicitly via the equivalence (\ref{derived:e}). 
In Proposition~\ref{prop:PDsim} and Proposition~\ref{prop:Persim}, 
we will give a complete description of these simple 
objects in $\PPer(X/Y)$. 

\subsection{Result for 3-folds II: Conjectural Bridgeland
stability conditions}
In three dimensional case, we wish
 to realize a story similar to Theorem~\ref{thm:intro1},
starting from $\PPer(X/Y)$
which corresponds to $\Coh(\mathrsfs{A})$
under the equivalence (\ref{derived:e}).  
However,
there is a serious issue in studying 
Bridgeland stability conditions on 3-folds: 
we don't even know whether there is a 
Bridgeland stability condition
on any projective 3-fold or not. 
In~\cite{BMT}, 
Bayer, Macri and the author 
constructed the heart of a bounded
 t-structure on any projective 
3-fold as a double tilting of $\Coh(X)$
which, together with a suitable central charge, 
conjecturally gives a point in the geometric 
chamber of the space of Bridgeland stability 
conditions. 
As an analogy of the work~\cite{BMT}, 
we construct a pair of the form
\begin{align}\label{pair:Y}
\sigma_{B, f^{\ast}\omega}=(Z_{B, f^{\ast}\omega}, \aA_{B, f^{\ast}\omega}).
\end{align}
Here the central charge $Z_{B, f^{\ast}\omega}$
is of the form (\ref{central}), 
replacing $\ch(E)$ by $e^{-B}\ch(E)$
where $B$ is a $\mathbb{Q}$-multiple of $D$. 
The heart $\aA_{B, f^{\ast}\omega}$ 
is constructed as a double tilting
of $\PPer(X/Y)$, similarly to 
the construction in~\cite{BMT}. 
We conjecture that the pair (\ref{pair:Y})
 gives a Bridgeland 
stability condition. 
In this case, 
the subcategory
\begin{align*}
\pP_{B, f^{\ast}\omega}(1) \cneq \{ E \in \aA_{B, f^{\ast}\omega} :
\Imm Z_{B, f^{\ast}\omega}(E)=0\}
\end{align*}
should be
the category of $\sigma_{B, f^{\ast}\omega}$-semistable 
objects of phase one. 
By denoting
$M^{\sigma_{B, f^{\ast}\omega}}([\oO_x])$
the set of isomorphism classes of $E \in \pP_{B, f^{\ast}\omega}(1)$
with $\ch(E)=\ch(\oO_x)$ for $x\in X$,
the result for 3-folds is formulated as follows: 
\begin{thm}\label{thm:intro2}
\emph{{\bf (Theorem~\ref{thm:3fold})  }}
Let $X$ be a smooth projective 3-fold and 
$f \colon X \to Y$ an extremal contraction.
Then 
there is a conjectural Bridgeland 
stability condition (\ref{pair:Y})
with
 $\pP_{B, f^{\ast}\omega}(1)$  
finite length\footnote{This means that 
$\pP_{B, f^{\ast}\omega}(1)$ is a noetherian and artinian 
abelian category.}
 abelian category,
such that
\begin{itemize}
\item If $f(D)$ is a curve, then $Y$ is 
one of the irreducible 
components of the coarse moduli space of 
$S$-equivalence classes of objects in 
$M^{\sigma_{B, f^{\ast}\omega}}([\oO_x])$. 
\item If $f(D)$ is a point, then 
$Y$ is the coarse moduli space of 
$S$-equivalence classes of objects in 
$M^{\sigma_{B, f^{\ast}\omega}}([\oO_x])$. 
\end{itemize}
\end{thm}

The finite length property of $\pP_{B, f^{\ast}\omega}(1)$
is a necessary condition for the pair (\ref{pair:Y})
to give a Bridgeland stability condition, 
and required to define $S$-equivalence classes of objects 
in $\pP_{B, f^{\ast}\omega}(1)$. 
Also when $f(D)$ is a curve, the coarse moduli 
space contains another component $f(D) \times f(D)$, 
so $Y$ appears only as an irreducible component. 
(cf.~Remark~\ref{rmk:glue}.)

The pair (\ref{pair:Y})
can be shown to give a Bridgeland 
stability condition if a conjectural 
BG type inequality evaluating $\ch_3$, 
similar to the one conjectured in~\cite{BMT}, 
holds. (cf.~Conjecture~\ref{conj:stab}.)
If this is true, then   
Theorem~\ref{thm:intro2}
realizes $Y$ as a moduli space of Bridgeland semistable objects. 
The conjectural stability 
condition (\ref{pair:Y})
should be in a boundary of the
conjectural geometric chamber 
constructed in~\cite{BMT}. 
So the pair (\ref{pair:Y}) 
is an analogue of $\sigma_{0}$ in Theorem~\ref{thm:intro1}. 

In types I and II, similarly to the surface case, 
 the variety $Y$ is likely to 
be a fine moduli space, if we are able to deform 
$\sigma_{B, f^{\ast}\omega}$. 
However in other cases, 
 the variety $Y$ may not be 
regarded as 
a fine moduli space, even if 
we deform $\sigma_{B, f^{\ast}\omega}$. 
 This is due to the fact that
the object $\dL f^{\ast} \oO_y$
is not an object in $D^b \Coh(X)$ for $y\in \mathrm{Sing}(Y)$. 

In order to show that 
the pair (\ref{pair:Y})
satisfies the desired property, 
we use a description of 
simple objects in $\PPer(X/Y)$, and investigate
the $S$-equivalence classes of 
objects in $\pP_{B, f^{\ast}\omega}(1)$. 
The most difficult and interesting case is 
the type V case, and the
arguments will be focused in this case. 

\subsection{Relation to existing works}
There are several recent works relating 
Bridgeland stability conditions and MMP.
In~\cite{ABCH}, the MMP of the Hilbert scheme of 
points in $\mathbb{P}^2$ is realized
as a variation of Bridgeland stability conditions
on $\mathbb{P}^2$. 
In~\cite{BaMa2}, the moduli spaces of generic 
Bridgeland stability 
conditions on K3 surfaces are shown to be 
projective varieties, and they are related 
by flops under variations of Bridgeland stability 
conditions. The motivation of this paper is similar to 
these works, but we concentrate on 
moduli spaces of objects whose numerical class
is equal to that of a skyscraper sheaf.

The Bridgeland stability conditions on arbitrary surfaces 
are constructed in~\cite{AB}. 
In the works~\cite{MYY}, \cite{MYY2}, \cite{YY}, 
\cite{Maci}, \cite{MaciMe}, 
the structure of walls and wall-crossing phenomena 
with respect to these stability conditions are studied. 
Our construction of $\sigma_0$
is a generalization of the construction in~\cite{AB}, 
and similar to the one in~\cite{MYY} for 
the perverse heart on K3 surfaces. 
Its deformation $\sigma_t$ for $t>0$
provides a
new example of Bridgeland stability 
conditions on arbitrary non-minimal surfaces.  
It would be interesting to study the moduli 
spaces of $\sigma_t$-semistable objects with 
arbitrary numerical classes, and 
investigate their behavior under wall-crossing. 

In the 3-fold case, the construction of $\eE$
in Theorem~\ref{thm:intro2} seems to 
have an application of the minimal saturated triangulated
category 
associated to the singular variety. 
In type V contraction, the direct summand of 
$\eE$ tensored by a line bundle provides a
local tilting generator of the saturated triangulated
category $\dD_Y$ constructed in~\cite{KawDer}. 
(cf.~Remark~\ref{rmk:satu}.)
It would be interesting to study $\dD_Y$
in terms of our vector bundle $\eE$.

\subsection{Plan of the paper}
In Section~\ref{sec:back}, we 
recall some background on Bridgeland stability 
conditions. In Section~\ref{sec:surface}, 
we prove Theorem~\ref{thm:intro1}. 
In Section~\ref{sec:3fold}, 
we prove Theorem~\ref{thm:intro0}
and classify simple objects 
in the perverse heart. 
In Section~\ref{sec:conj},
we construct a conjectural 
Bridgeland stability condition and 
prove Theorem~\ref{thm:intro2}.

\subsection{Acknowledgement}
The author is greatful to Arend Bayer for 
valuable comments. 
The proof of Proposition~\ref{prop:support} is benefited
by the communication with Emanuele Macri. 
This work is supported by World Premier 
International Research Center Initiative
(WPI initiative), MEXT, Japan. This work is also supported by Grant-in Aid
for Scientific Research grant (22684002), 
and partly (S-19104002),
from the Ministry of Education, Culture,
Sports, Science and Technology, Japan.

\subsection{Notation and convention}
In this paper, all the varieties are defined over 
$\mathbb{C}$. 
For a triangulated (or abelian)
category $\dD$ and 
a set of objects $\sS \subset \dD$, 
we denote by $\langle \sS \rangle$ 
the smallest extension closed subcategory of 
$\dD$ which contains objects in $\sS$. 
For a sheaf of (not necessary commutative) 
$\oO_Y$-algebras $\mathrsfs{A}$ on a 
variety $Y$, we denote by 
$\Coh(\mathrsfs{A})$ the abelian category of 
coherent right $\mathrsfs{A}$-modules on $Y$. 
The subcategory $\Coh_{\le i}(\mathrsfs{A}) \subset \Coh(\mathrsfs{A})$
is the category of $E\in \Coh(\mathrsfs{A})$ with 
$\dim \Supp(E) \le 1$ as $\oO_Y$-module. 
For $E\in D^b \Coh(X)$, we denoted by 
$\hH^i(E)$ the $i$-th 
cohomology sheaf of the complex $E$.

\section{Background on stability conditions}\label{sec:back}
In this section, we review the theory of 
stability conditions by Bridgeland~\cite{Brs1}.
\subsection{Definitions}
Let $X$ be a smooth projective variety and $N(X)$ the 
numerical Grothendieck group of $X$. 
This is the quotient of the usual Grothendieck 
group $K(X)$ by the subgroup of $E\in K(X)$
with $\chi(E, F)=0$ for any $F\in K(X)$, 
where $\chi(E, F)$ is the Euler pairing
\begin{align*}
\chi(E, F) \cneq \sum_{i\in \mathbb{Z}}
(-1)^i \dim \Ext^i(E, F). 
\end{align*}
\begin{defi}\label{lem:pair2} \emph{(\cite{Brs1})}
A stability condition\footnote{A stability condition in 
Definition~\ref{lem:pair2}
 was called \textit{numerical} stability condition in~\cite{Brs1}.
We omit `numerical' since we do not deal with 
non-numerical stability conditions.} on $X$ is a pair 
\begin{align}\label{pair2}
(Z, \aA), \quad \aA \subset D^b \Coh(X),
\end{align}
where $Z \colon N(X) \to \mathbb{C}$
is a group homomorphism and $\aA$ is the heart of a 
bounded t-structure, 
such that the following conditions hold: 
\begin{itemize}
\item For any non-zero $E\in \aA$, we have 
\begin{align}\label{pro:1}
Z(E) \in \{ r\exp(i\pi \phi) : r>0, \phi \in (0, 1] \}.
\end{align}
\item (Harder-Narasimhan property)
For any $E\in \aA$, there is a filtration in $\aA$
\begin{align*}
0=E_0 \subset E_1 \subset \cdots \subset E_N
\end{align*}
such that each subquotient $F_i=E_i/E_{i-1}$ is 
$Z$-semistable with 
$\arg Z(F_i)> \arg Z(F_{i+1})$. 
\end{itemize}
\end{defi}
Here an object $E\in \aA$ is $Z$-\textit{(semi)stable} 
if for any subobject $0\neq F \subsetneq E$
we have 
\begin{align*}
\arg Z(F) <(\le) \arg Z(E). 
\end{align*}
The group homomorphism $Z$ is called a \textit{central charge}.
In this paper,
the central charge is always of the form $Z=Z_{B, \omega}$: 
\begin{align}\label{integral}
Z_{B, \omega}(E) &= -\int_{X} e^{-i\omega} \ch^B(E).
\end{align}
 Here  
$B$, $\omega$ are elements of $H^2(X, \mathbb{R})$
and 
\begin{align*}
\ch^B(E) \cneq e^{-B}\ch(E) \in H^{\ast}(X, \mathbb{R})
\end{align*} is the twisted 
Chern character.
If $B=0$, we just write $Z_{\omega} \cneq Z_{0, \omega}$. 
By setting $d=\dim X$, $Z_{B, \omega}(E)$ is written as 
\begin{align*}
\sum_{j \ge 0}\frac{(-1)^{j+1}}{(2j)!}\omega^{2j}\ch^B_{d-2j}(E)
 + \sqrt{-1} \left(	
\sum_{j\ge 0}\frac{(-1)^j}{(2j+1)!} \omega^{2j+1}\ch^B_{d-2j-1}(E)\right).
\end{align*}
\subsection{Slicing}
Given a stability condition $\sigma=(Z, \aA)$, 
we can construct 
subcategories $\pP(\phi) \subset D^b \Coh(X)$
for $\phi \in \mathbb{R}$ as follows: 
if $0<\phi \le 1$, 
then $\pP(\phi)$ is defined by 
\begin{align*}
\pP(\phi) \cneq \left\{ E \in \aA : \begin{array}{c}
E \mbox{ is } Z \mbox{-semistable with } \\
Z(E) \in \mathbb{R}_{>0} \exp(i\pi \phi) 
\end{array} \right\} \cup \{0\}. 
\end{align*}
Other $\pP(\phi)$ are defined by the rule 
\begin{align*}
\pP(\phi+1)=\pP(\phi)[1].
\end{align*} 
A family of subcategories $\{\pP(\phi) \}_{\phi \in \mathbb{R}}$
forms a \textit{slicing} in the sense of~\cite[Definition~3.3]{Brs1}.
As proved in~\cite[Proposition~5.3]{Brs1}, 
giving a stability condition is equivalent to giving a pair 
\begin{align}\label{stab:slice}
(Z, \{\pP(\phi)\}_{\phi \in \mathbb{R}}),
\end{align}
where $\{\pP(\phi)\}_{\phi \in \mathbb{R}}$
is a slicing, $Z \colon N(X) \to \mathbb{C}$ is a group homomorphism
satisfying a certain axiom. (cf.~\cite[Definition~1.1]{Brs1}.)

A non-zero object $E \in \pP(\phi)$ is called 
$\sigma$-\textit{semistable of phase}
$\phi$. 
Each subcategory $\pP(\phi) \subset D^b \Coh(X)$
is shown to be an abelian category, and a
simple object in $\pP(\phi)$ is called
$\sigma$-\textit{stable}.
For an interval $I\subset \mathbb{R}$, 
we set 
\begin{align*}
\pP(I) \cneq \langle \pP(\phi) : \phi \in I \rangle.
\end{align*}
A stability condition (\ref{stab:slice})
is called \textit{locally finite} if 
for any $\phi \in \mathbb{R}$, 
there is $\epsilon >0$ so that $\pP((\phi-\epsilon, \phi+\epsilon))$
is of finite length, i.e. 
it is a noetherian and artinian 
with respect to strict epimorphisms
and strict monomorphisms. (cf.~\cite[Definition~5.7]{Brs1}.) 
In particular each $\pP(\phi)$ is a finite length 
abelian category. 

In general if $\pP$ is a finite length 
abelian category, 
 any object $F \in \pP$
admits a Jordan-H$\rm{\ddot{o}}$lder filtration in $\pP$
\begin{align}\label{Jordan}
0=F_0 \subset F_1 \subset \cdots \subset F_N=F
\end{align}
so that each subquotient 
$F_i/F_{i-1}$ is simple. 
We set 
\begin{align}\label{gr}
\mathrm{gr}(F) \cneq \bigoplus_{i=1}^{N} F_i/F_{i-1} \in \pP. 
\end{align}
Although the filtration (\ref{Jordan}) is not 
unique, the object $\mathrm{gr}(F)$ is uniquely 
determined. 
\begin{defi}\label{S-eq}
Let $\pP$
be a finite length abelian category. 
Then $F, F' \in \pP$
is called $S$-equivalent if there 
is an isomorphism $\mathrm{gr}(F) \cong \mathrm{gr}(F')$. 
\end{defi}

\subsection{The space of stability conditions}
Let $\lVert \ast \rVert$ be a fixed
norm on $N(X)_{\mathbb{R}}$. 
It is useful (but often difficult to check)
to put the following additional 
condition on the pair (\ref{stab:slice}): 
\begin{defi}
A stability condition (\ref{stab:slice})
satisfies the support property if there 
is a positive constant $C>0$ such 
that the following inequality holds for any 
non-zero $E\in \cup_{\phi \in \mathbb{R}}\pP(\phi)$: 
\begin{align*}
\frac{\lVert E \rVert}{\lvert Z(E) \rvert} \le C. 
\end{align*}
\end{defi}

The space of stability conditions is
defined as follows: 
\begin{defi}\label{def:stabsp}
(i) We define $\Stab^{\dag}(X)$ to 
be the set of locally finite
 stability conditions on $D^b \Coh(X)$. 

(ii) 
We define $\Stab(X)$ to be the set of 
stability conditions 
on $D^b \Coh(X)$ satisfying the support property. 
\end{defi}

It is easy to see that the local finiteness of 
the stability condition follows from the support property, 
i.e. we have the inclusion
\begin{align*}
\Stab(X) \subset \Stab^{\dag}(X).
\end{align*}
As we will mention in Remark~\ref{rmk:locfin}, 
the set $\Stab(X)$ has the better property than $\Stab^{\dag}(X)$. 
The following is Bridgeland's main theorem.
\begin{thm}\label{thm:loc}\emph{(\cite[Theorem~1.2]{Brs1})}
There is a natural topology on $\Stab(X)$ such that 
the forgetting map 
\begin{align*}
\Stab(X) \to N(X)_{\mathbb{C}}^{\vee}
\end{align*}
sending $(Z, \aA)$ to $Z$
is a local homeomorphism. 
In particular, each connected component of $\Stab(X)$
is a complex manifold. 
\end{thm}
\begin{rmk}\label{rmk:locfin}
The set $\Stab^{\dag}(X)$ was 
considered in the original paper~\cite{Brs1}, and 
there is also a natural topology on $\Stab^{\dag}(X)$. 
However the forgetting map 
$\Stab^{\dag}(X) \to N(X)^{\vee}_{\mathbb{C}}$
may not be a local homeomorphism, but so 
on a certain linear subspace of $N(X)^{\vee}_{\mathbb{C}}$. 
The subspace $\Stab(X) \subset \Stab^{\dag}(X)$ is 
shown to be the union of connected components on which the 
forgetting map is a local homeomorphism. 
\end{rmk}
\subsection{Bogomolov-Gieseker inequality}
We use the following 
(generalized) Bogomolov-Gieseker 
inequality to 
show the axiom (\ref{pair2}) or support property.  
\begin{thm}\label{thm:BoGi} \emph{(\cite{Bog},~\cite{Gie}, \cite{Langer})}
Let $X$ be a smooth projective variety
of $\dim X=d \ge 2$, and 
$H_1, \cdots, H_{d-1}$ nef divisors on $X$
such that the 1-cycle $H_1 \cdots H_{d-1}$ is 
numerically nontrivial. 
Then for any torsion free
$(H_1, \cdots, H_{d-1})$-slope semistable sheaf $E$ on $X$, we have the 
inequality
\begin{align*}
(\ch_1(E)^2 -2\ch_0(E) \ch_2(E))H_2 \cdots H_{d-1} \ge 0. 
\end{align*}
\end{thm}

\section{Extremal contractions of projective surfaces}\label{sec:surface}
The goal of this section is to prove Theorem~\ref{thm:intro1}. 
Throughout this section, 
$X$ is a smooth projective surface
and $f$ is a birational morphism 
\begin{align*}
f \colon X \to Y
\end{align*}
which contacts a single $(-1)$-curve
$C \subset X$ to a point in $Y$. 
\subsection{Perverse t-structure on $D^b \Coh(X)$}
We recall the construction of 
the hearts of perverse t-structures 
\begin{align*}
\ppPPer(X/Y) \subset D^b \Coh(X)
\end{align*}
for $p \in \mathbb{Z}$, 
studied in~\cite{Br1}, \cite{MVB}. 
Let $\cC$ be the triangulated subcategory 
of $D^b \Coh(X)$, defined by 
\begin{align*}
\cC \cneq \{ E \in D^b \Coh(X) : 
\dR f_{\ast}E=0\}. 
\end{align*}
Since the dimensions of the 
 fibers of $f$ are
at most one dimensional, 
the standard 
t-structure on $D^b \Coh(X)$ induces 
a t-structure on $\cC$,
\begin{align*}
(\cC^{\le 0}, \cC^{\ge 0})
\end{align*}
with heart $\cC^0 =\cC \cap \Coh(X)$. 
(cf.~\cite[Lemma~3.1]{Br1}.)
\begin{defi}\emph{(\cite{Br1}, \cite{MVB})}
We define $\ppPPer(X/Y)$ to be 
\begin{align*}
\ppPPer(X/Y) \cneq \left\{ 
E \in D^b \Coh(X) : 
\begin{array}{c}
\dR f_{\ast} E \in \Coh(Y), \\
\Hom(\cC^{<p}, E)=\Hom(E, \cC^{>p})=0 
\end{array}
\right\}. 
\end{align*}
\end{defi}
In what follows, we mainly use the case $p=-1$, 
and denote 
\begin{align*}
\PPer(X/Y) \cneq \iPPer(X/Y).
\end{align*}
Also we denote 
by $\hH_{p}^{i}(\ast)$ the $i$-th cohomology 
functor with respect to the t-structure 
with heart $\PPer(X/Y)$. 

As proved in~\cite{Brs1}, \cite{MVB}, 
the subcategory $\PPer(X/Y)$ is the 
heart of a bounded t-structure on $D^b \Coh(X)$.
Indeed, Van den Bergh~\cite{MVB}
shows that 
the t-structure $\PPer(X/Y)$ is related
to a sheaf of non-commutative algebras on $Y$. 
Let us set
\begin{align*}
\eE \cneq \oO_X \oplus \oO_X(-C), \
\mathrsfs{A} \cneq f_{\ast} \eE nd(\eE).
\end{align*}
Note that $\mathrsfs{A}$ is a sheaf of 
non-commutative algebras on $Y$. 
\begin{thm}\emph{(\cite{MVB})}\label{thm:nc}
We have the derived equivalence
\begin{align*}
\Phi \cneq \dR f_{\ast}
\dR \hH om(\eE, \ast) :
D^b \Coh(X) \stackrel{\sim}{\to} D^b \Coh(\mathrsfs{A}),
\end{align*}
which restricts to an equivalence between 
$\PPer(X/Y)$ and $\Coh(\mathrsfs{A})$. 
\end{thm}
We define the following subcategory of $\PPer(X/Y)$:
\begin{align}\label{per:inc}
\PPer_{\le i}(X/Y) \cneq 
\{ E \in \PPer(X/Y) : \Phi(E) \in \Coh_{\le i}(\mathrsfs{A})\},
\end{align}
and write $\PPer_{0}(X/Y) \cneq \PPer_{\le 0}(X/Y)$. 
\begin{rmk}
By the construction of $\eE$, 
an object $E\in \PPer(X/Y)$ is an object in 
$\PPer_{0}(X/Y)$ iff $\Supp(E)$ is zero dimensional 
outside $C$, and an object in $\PPer_{\le 1}(X/Y)$
iff $\Supp(E)$ is at most one dimensional. 
\end{rmk}
We collect some known results used in this paper.
\begin{prop}\label{prop:collect}
(i) The category $\PPer_{0}(X/Y)$ is a finite 
length abelian category with simple objects
given by 
\begin{align*}
\oO_C, \ \oO_C(-1)[1], \ \oO_x, \ x \in X \setminus C. 
\end{align*} 

(ii) An object $E \in D^b \Coh(X)$ is an object
in $\PPer(X/Y)$ if and only if 
$\hH^i(E)=0$ for $i\neq 0, -1$, 
$\Hom(\hH^{0}(E), \oO_C(-1))=0$ and 
\begin{align*}
R^1 f_{\ast}\hH^0(E)=f_{\ast} \hH^{-1}(E)=0. 
\end{align*}
\end{prop}
\begin{proof}
The result of (i) follows from~\cite[Proposition~3.5.8]{MVB}.
The result of (ii) follows from~\cite[Lemma~3.2]{Br1}.
\end{proof}

\subsection{Slope stability on $\PPer(X/Y)$}
Let $\omega$ be an ample divisor on $Y$. 
Similarly to the usual slope 
function on coherent sheaves, 
we define the slope function $\mu_{f^{\ast}\omega}$
 on $\PPer(X/Y)$
as follows: 
if $E \in \PPer(X/Y)$ satisfies $\ch_0(E)>0$, we set 
\begin{align*}
\mu_{f^{\ast}\omega}(E) \cneq \frac{\ch_1(E) \cdot f^{\ast}\omega}{\ch_0(E)}.
\end{align*}
Otherwise we set $\mu_{f^{\ast}\omega}(E)=\infty$. 
Similarly to the usual slope stability, 
the slope $\mu_{f^{\ast}\omega}$ satisfies the 
weak seesaw property, and 
defines the weak stability condition on 
$\PPer(X/Y)$ as follows:
\begin{defi}\label{defi:slope}
An object $E\in \PPer(X/Y)$ is 
$\mu_{f^{\ast}\omega}$-(semi)stable 
if for any exact sequence $0 \to F \to E \to G \to 0$
in $\PPer(X/Y)$, we have the inequality
\begin{align*}
\mu_{f^{\ast}\omega}(F) <(\le) \mu_{f^{\ast}\omega}(G). 
\end{align*}
\end{defi}
Similarly to the usual slope stability, 
we have the following:
\begin{lem}\label{lem:HN}
Any object in $\PPer(X/Y)$ admits a 
Harder-Narasimhan filtration 
with respect to $\mu_{f^{\ast}\omega}$-stability.
\end{lem}
\begin{proof}
By Theorem~\ref{thm:nc}, the 
category $\PPer(X/Y)$ is noetherian. 
Therefore it is enough to check that 
there is no infinite sequence
\begin{align*}
E_1 \supset E_2 \supset \cdots \supset E_{i} \supset \cdots
\end{align*}
such that $\mu_{f^{\ast}\omega}(E_{i+1})> \mu_{f^{\ast}\omega}(E_i/E_{i+1})$
for all $i$. (cf.~\cite[Proposition~2.12]{Tcurve1}.)
Suppose that there is such a sequence. 
Since $\ch_0(\ast)$ is non-negative on $\PPer(X/Y)$, 
we may assume $\ch_0(E_i)$ is constant. 
But then $\mu_{f^{\ast}\omega}(E_i/E_{i+1})=\infty$, 
which is a contradiction. 
\end{proof}

\subsection{Bogomolov-Gieseker inequality for perverse 
coherent sheaves}
We are going to construct a Bridgeland 
stability condition 
on $D^b \Coh(X)$ by using the 
 tilting of $\PPer(X/Y)$. 
For this purpose, 
we need to evaluate 
Chern classes of $\mu_{f^{\ast}\omega}$-semistable 
perverse coherent sheaves. 
First we show the following lemma: 
\begin{lem}\label{lem:Cch}
Let $\aA \subset D^b \Coh(X)$ be the heart 
of a bounded t-structure with 
$\oO_C, \oO_C(-1)[1] \in \aA$. 
For $E\in \aA$, suppose that 
$\Hom(\oO_C, E)=0$. Then we have 
$C \cdot \ch_1(E) \ge 0$. 
\end{lem}
\begin{proof}
By the Serre duality, we have 
\begin{align*}
\Ext^i(\oO_C, E) &\cong \Hom^{1-i}(E, \oO_C(-1)[1])^{\vee} \\
&\cong 0,
\end{align*}
for $i\ge 2$. 
Here the last isomorphism follows
from that both of $E$ and $\oO_C(-1)[1]$ are objects in 
the heart $\aA$. Also by the assumption, 
we have 
\begin{align*}
\Hom^{\le 0}(\oO_C, E)=0. 
\end{align*}
By the above vanishing and 
the Riemann-Roch theorem, 
we have 
\begin{align*}
-C \cdot \ch_1(E) &= \chi(\oO_C, E) \\
&= -\dim \Ext^1(\oO_C, E).
\end{align*}
Therefore $C \cdot \ch_1(E) \ge 0$ holds. 
\end{proof}
Another lemma we use is the following:
\begin{lem}\label{lem:per0}
For $T \in \PPer_{0}(X/Y)$,
 we have the inequality
\begin{align*}
\ch_2(T) \ge \frac{1}{2} C\cdot \ch_1(T).
\end{align*}
\end{lem}
\begin{proof}
By the definition of $\PPer_{0}(X/Y)$, the 
object $\dR f_{\ast}T$ is a zero dimensional sheaf on $Y$. 
Therefore we have 
\begin{align*}
\chi(\oO_X, T) = \chi(\oO_Y, \dR f_{\ast}T)  \ge 0. 
\end{align*}
On the other hand, the Riemann-Roch theorem implies 
\begin{align*}
\chi(\oO_X, T)= \ch_2(T)-\frac{1}{2}C \cdot \ch_1(T). 
\end{align*}
Therefore the desired inequality holds. 
\end{proof}
The following is the 
Bogomolov-Gieseker inequality 
for perverse coherent sheaves. 
\begin{prop}\label{prop:BGper}
For any $\mu_{f^{\ast}\omega}$-semistable 
$E\in \PPer(X/Y)$ with $\ch_0(E)>0$, we have the 
inequality
\begin{align}\label{BG:per}
\ch_1(E)^2 -2\ch_0(E) \ch_2(E) \ge 0. 
\end{align}
\end{prop}
\begin{proof}
Since $E \in \PPer(X/Y)$ is $\mu_{f^{\ast}\omega}$-semistable
and 
$\hH^{-1}(E)[1]$ is a subobject of $E$ in $\PPer(X/Y)$
supported on $C$ by Proposition~\ref{prop:collect} (ii), 
we must have $\hH^{-1}(E)=0$. 
Hence $E$ is a sheaf on $X$ but it may have a torsion. 
Let us take an exact sequence of sheaves
\begin{align}\label{TEF}
0 \to T \to E \to F \to 0,
\end{align}
where $T$ is a torsion sheaf and $F$ is a torsion free sheaf. 
The above exact sequence
and Proposition~\ref{prop:collect} (ii) 
 easily imply that $F \in \PPer(X/Y)$. 
We have the exact sequence
in $\PPer(X/Y)$
\begin{align*}
0 \to \hH_p^0(T) \to E \to F \to \hH_p^1(T) \to 0,
\end{align*}
and $\hH^i_p(T)=0$ for $i\neq 0, 1$. 
By the $\mu_{f^{\ast}\omega}$-stability 
of $E$, we have $\hH_p^0(T)=0$ hence 
$T[1] \in \PPer(X/Y)$ follows. 
This also implies that $f_{\ast}T =0$, 
hence $T$ is supported on $C$. 

Since sheaves supported on $C$ do not affect 
$\mu_{f^{\ast}\omega}(\ast)$, it follows 
that $F$ is a $\mu_{f^{\ast}\omega}$-semistable 
sheaf in the category of coherent sheaves. 
Therefore Theorem~\ref{thm:BoGi} 
implies
\begin{align}\label{ineq:F}
\ch_1(F)^2 -2\ch_0(F) \ch_2(F) \ge 0. 
\end{align}

Let us write 
\begin{align*}
\ch(T)=(0, aC, \ch_2(T)) \in H^0(X) \oplus H^2(X) \oplus H^4(X)
\end{align*}
for $a \in \mathbb{Z}_{\ge 0}$. 
By Lemma~\ref{lem:per0}, we have $\ch_2(T) \le -a/2$. 
Also Lemma~\ref{lem:Cch} implies 
\begin{align}\label{ineq:chE}
C \cdot \ch_1(E)= C \cdot \ch_1(F) -a \ge 0. 
\end{align} 
By combining these inequalities, we have
\begin{align*}
&\ch_1(E)^2 -2\ch_0(E) \ch_2(E) \\
&= (\ch_1(F)+aC)^2 -2\ch_0(F) (\ch_2(F) + \ch_2(T)) \\
&= \ch_1(F)^2 -2\ch_0(F) \ch_2(F) +2aC \cdot \ch_1(F) -a^2 -2\ch_0(F) \ch_2(T) \\&\ge a\left( \ch_0(F) + 2C \cdot \ch_1(F) -a  \right) \\
&\ge 0. 
\end{align*}
Here we have used (\ref{ineq:F}), (\ref{ineq:chE}) and $\ch_2(T) \le -a/2$
in the third inequality. 
Hence the inequality (\ref{BG:per}) is proved. 
\end{proof}
\begin{rmk}\label{rmk:weak}
By the Hodge index theorem, the inequality 
(\ref{BG:per})
also implies 
\begin{align}\label{weak}
(\ch_1(E) \cdot f^{\ast}\omega)^2 \ge 2\omega^2 \ch_0(E) \ch_2(E). 
\end{align}
However the weaker inequality (\ref{weak}) is much easier to prove: 
it immediately follows from the exact sequence (\ref{TEF})
and the inequality (\ref{ineq:F}), 
together with $\ch_2(T) \le -a/2 \le 0$. 
The weaker inequality (\ref{weak}) will be used in 
the construction of a stability condition, while the stronger
 one (\ref{BG:per}) will be required to show the support property. 
\end{rmk}

\subsection{Construction of $\sigma_{0}$}\label{subsec:const}
In this subsection, we construct a stability 
condition $\sigma_0$
on $D^b \Coh(X)$ via tilting of $\PPer(X/Y)$. 
Let $(\tT_{f^{\ast}\omega}, \fF_{f^{\ast}\omega})$
 be the pair of subcategories of $\PPer(X/Y)$
defined by 
\begin{align*}
\tT_{f^{\ast}\omega} &\cneq \langle E : E \mbox{ is }\mu_{f^{\ast}\omega}
\mbox{-semistable with } \mu_{f^{\ast}\omega}(E)>0 \rangle, \\
\fF_{f^{\ast}\omega} &\cneq \langle E : E \mbox{ is }\mu_{f^{\ast}\omega}
\mbox{-semistable with } \mu_{f^{\ast}\omega}(E) \le 0 \rangle. 
\end{align*}
The pair of subcategories $(\tT_{f^{\ast}\omega}, \fF_{f^{\ast}\omega})$
 forms a torsion pair~\cite{HRS}
on $\PPer(X/Y)$
by Lemma~\ref{lem:HN}.
The associate tilting is defined in the following way: 
\begin{defi}
We define $\bB_{f^{\ast}\omega} \subset D^b \Coh(X)$ to be 
\begin{align*}
\bB_{f^{\ast}\omega} \cneq \langle \fF_{f^{\ast}\omega}[1], 
\tT_{f^{\ast}\omega} \rangle. 
\end{align*}
\end{defi}
As for the central charge, 
we take $Z_{f^{\ast}\omega}$
in the notation of (\ref{integral}). 
It is written as 
\begin{align*}
Z_{f^{\ast}\omega}(E) 
= \left(-\ch_2(E) + \frac{\omega^2}{2} \ch_0(E)  \right) 
+ \sqrt{-1} \ch_1(E) \cdot f^{\ast}\omega. 
\end{align*}
We show the following lemma: 
\begin{lem}
The pair $(Z_{f^{\ast}\omega}, \bB_{f^{\ast}\omega})$
determines a stability condition on $D^b \Coh(X)$.
\end{lem}
\begin{proof}
We first check the property (\ref{pro:1}). 
Let us take a non-zero object $E \in \bB_{f^{\ast}\omega}$. 
By the construction of $\bB_{f^{\ast}\omega}$, 
we have $\Imm Z_{f^{\ast}\omega}(E) \ge 0$. 
Suppose that $\Imm Z_{f^{\ast}\omega}(E)=0$. 
Then we have 
\begin{align}\label{Imm=0}
E \in \left\langle F[1], T : \begin{array}{c}
F \in \PPer(X/Y) 
\mbox{ is } \mu_{f^{\ast}\omega} \mbox{-semistable with } \\
\mu_{f^{\ast}\omega}(F)=0, \ T \in \PPer_{0}(X/Y).
\end{array} \right\rangle. 
\end{align}
It is enough to check (\ref{pro:1}) 
for $E=F[1]$ where $F$ is a
$\mu_{f^{\ast}\omega}$-semistable 
object in $\PPer(X/Y)$ or $E \in \PPer_{0}(X/Y)$. 
In the former case, 
the inequality in Remark~\ref{rmk:weak}
implies $\ch_2(F) \le 0$. 
Combined with $\ch_0(F)>0$, we have 
$Z_{f^{\ast}\omega}(F[1]) \in \mathbb{R}_{<0}$. 
In the latter case, noting 
Proposition~\ref{prop:collect} (i), 
the property (\ref{pro:1}) follows from 
\begin{align*}
Z_{f^{\ast}\omega}(\oO_C)=Z_{f^{\ast}\omega}(\oO_C(-1)[1])=-\frac{1}{2},
\end{align*}
and $Z_{f^{\ast}\omega}(\oO_x)=-1$ for $x\in X \setminus C$. 

As for the Harder-Narasimhan property, we 
will give a 3-fold version of a similar statement in Lemma~\ref{B:HN}. 
The same proof is applied. 
\end{proof}
Below we fix an ample divisor $\omega$ on $Y$ and set 
\begin{align}\label{sigma0}
\sigma_0 \cneq (Z_{f^{\ast}\omega}, \bB_{f^{\ast}\omega}).
\end{align}
Note that, at this moment, 
$\sigma_0$ is just an element of $\Stab^{\dag}(X)$
\footnote{The locally finiteness of $\sigma_0$
is obvious since 
the image of $Z_{f^{\ast}\omega}$ is a discrete subgroup.}
in the notation of Definition~\ref{def:stabsp}. 
The following result will be required 
in order to deform $\sigma_0$:
\begin{prop}\label{prop:support}
We have $\sigma_0 \in \Stab(X)$, i.e. 
$\sigma_0$ satisfies the support property. 
\end{prop}
The proof of Proposition~\ref{prop:support}
is postponed until Subsection~\ref{subsec:support}.

\subsection{One parameter family of stability conditions}
By Proposition~\ref{prop:support} and Theorem~\ref{thm:loc}, 
a small deformation of $Z_{f^{\ast}\omega}$ 
uniquely lifts to a small deformation of $\sigma_0$. 
Therefore for $0<\varepsilon \ll 1$, 
 there is a unique continuous map
\begin{align*}
 \sigma \colon (-1, 1)  \to \Stab(X)
\end{align*}
such that $\sigma_t \cneq \sigma(t)$ satisfies 
the following: 
$\sigma_0$ coincides with the stability condition (\ref{sigma0}), 
and other $\sigma_t$ is written as
\begin{align*}
\sigma_t=(Z_{f^{\ast}\omega +\varepsilon tC}, \pP_t),
\end{align*}
where $\pP_t=\{\pP_t(\phi)\}_{\phi \in \mathbb{R}}$ is a slicing 
on $D^b \Coh(X)$. Let 
$M^{\sigma_t}([\oO_x])$ be the set of isomorphism 
classes of objects $E \in \pP_t(1)$ with 
$\ch(E)=\ch(\oO_x)$ for $x\in X$. We have 
the following proposition:
\begin{prop}\label{prop:wall}
By taking $\varepsilon$ smaller if necessary, we have the following:

(i) If $t<0$, we have $M^{\sigma_t}([\oO_x])=\{ \oO_x : x\in X\}$. 

(ii) If $t>0$, we have $M^{\sigma_t}([\oO_x])=\{ \dL f^{\ast} \oO_y : 
y\in Y\}$.

(iii) If $t=0$, we have 
\begin{align}\label{M:wall}
M^{\sigma_0}([\oO_x])=
\{ \oO_x, \dL f^{\ast}\oO_0, \oO_C \oplus \oO_C(-1)[1] : x\in X\}.
\end{align} 
Here $0=f(C) \in Y$. 
\end{prop}
\begin{proof}
We first show (iii). Let us take 
$E\in \bB_{f^{\ast}\omega}$ with 
$\ch(E)=\ch(\oO_x)$. 
Since $\Imm Z_{f^{\ast}\omega}(E)=0$, 
$E$ is contained in the RHS of (\ref{Imm=0}), 
hence the condition $\ch(E)=\ch(\oO_x)$
implies $E \in \PPer_{0}(X/Y)$.
By Proposition~\ref{prop:collect} (i), we have the following 
possibilities: $E \cong \oO_x$ for $x\in X \setminus C$
and it is $\sigma_0$-stable, or 
$E$ is not $\sigma_0$-stable with 
\begin{align*}
\mathrm{gr}(E) \cong \oO_C \oplus \oO_C(-1)[1].
\end{align*} 
Here $\mathrm{gr}(E)$ is defined by (\ref{gr}). 
In the latter case, if $E$ is not isomorphic 
to $\oO_C \oplus \oO_C(-1)[1]$, 
$E$ fits into one of the following non-split 
exact sequences in $\PPer_{0}(X/Y)$: 
\begin{align}
\label{one:1}
&0 \to \oO_C \to E \to \oO_C(-1)[1] \to 0, \\
\label{one:2}
&0 \to \oO_C(-1)[1] \to E \to \oO_C \to 0. 
\end{align}
In the case of (\ref{one:1}), 
we have $E \cong \oO_x$ for $x\in C$. 
In the case of (\ref{one:2}), we have 
$E \cong \dL f^{\ast} \oO_0$. 
Indeed, it is easy to 
check that there is a distinguished triangle
\begin{align}\label{one:3}
\oO_C(-1)[1] \to \dL f^{\ast} \oO_0 \to \oO_C,
\end{align}
hence $E \cong \dL f^{\ast}\oO_0$
follows since $\Ext^1_{X}(\oO_C, \oO_C(-1)[1])=\mathbb{C}$. 

Next we show (i) and (ii). 
The support property of $\sigma_0$ 
ensures the existence of a wall and chamber structure 
on $\Stab(X)$ near $\sigma_0$, i.e. 
the set of semistable objects $E$ with $\ch(E)=\ch(\oO_x)$
is constant on a chamber. Therefore, by choosing $\varepsilon$
smaller if necessary, we may assume that 
$M^{\sigma_t}([\oO_x])$ is constant 
for $t\in (-1, 0)$ and $t\in (0, 1)$. 
Since the set of points in $\Stab(X)$ in 
which a given object is semistable is closed, 
we have the inclusion $M^{\sigma_t}([\oO_x]) \subset M^{\sigma_0}([\oO_x])$.
It is enough to check $\sigma_t$-stability for 
the objects in the RHS of (\ref{M:wall}).  

Since $\oO_x$ for $x\in X\setminus C$, $\oO_C$ and 
$\oO_C(-1)[1]$ are stable in $\sigma_0$, they 
are also stable in $\sigma_t$ for any $t\in (-1, 1)$. 
Let us take real numbers 
$\phi_{t}, \psi_t$ 
so that the following holds:
\begin{align*}
\oO_C \in \pP_t(\phi_t), \quad \oO_C(-1)[1] \in \pP_t(\psi_t).
\end{align*}
We have $\phi_t<\psi_t$ for $t<0$ and $\phi_t>\psi_t$ for $t>0$. 
Therefore, by looking at the exact sequence (\ref{one:1})
for $E =\oO_x$ with $x\in C$, we see that 
 $E$ is $\sigma_t$-stable for $t<0$ but not 
$\sigma_t$-semistable for $t>0$. 
Similarly by the sequence (\ref{one:3}), for $E=\dL f^{\ast} \oO_0$, 
we see that $E$ is $\sigma_t$-stable for $t>0$
but not $\sigma_t$-semistable for $t<0$. 
Finally, it is obvious that 
$\oO_C \oplus \oO_C(-1)[1]$ is not 
$\sigma_t$-semistable for $t\neq 0$. 
Therefore (i) and (ii) are proved. 
\end{proof}
\begin{rmk}
The proof of the above proposition 
shows that $M_t^{\sigma_t}([\oO_x])$ consists
of $\sigma_t$-stable objects when $t\neq 0$. 
\end{rmk}
We now give a proof of Theorem~\ref{thm:intro1}. 
\begin{thm}\label{thm:main1}
We have the following: 

(i) If $t<0$, then $X$ is the fine moduli 
space of $\sigma_t$-stable objects in $M^{\sigma_t}([\oO_x])$. 

(ii) If $t=0$, then $Y$ is the 
coarse moduli space of $S$-equivalence
classes of objects in $M^{\sigma_0}([\oO_x])$. 

(iii) If $t>0$, then $Y$ is the fine moduli 
space of $\sigma_t$-stable objects in $M^{\sigma_t}([\oO_x])$.
\end{thm}
\begin{proof}
Let $\mM$ be the algebraic 
space 
which parameterizes
objects $E\in D^b \Coh(X)$
satisfying
\begin{align*}
\Ext^{<0}(E, E)=0, \ \Hom(E, E)=\mathbb{C},
\end{align*}
constructed by Inaba~\cite{Inaba}.
We have the open sub algebraic space
$\mM^{\sigma_t}([\oO_x]) \subset \mM$
whose closed points correspond to 
$\sigma_t$-stable objects in $M^{\sigma_t}([\oO_x])$. 
If $t>0$, Proposition~\ref{prop:wall} (ii) implies 
that there is a morphism 
\begin{align}\label{mor:Y}
Y \to \mM^{\sigma_t}([\oO_x])
\end{align}
sending $y \in Y$ to the point corresponding to 
 $\dL f^{\ast} \oO_y$, which is 
bijective on $\mathbb{C}$-valued points. 
Also since we have the 
fully faithful functor 
\begin{align*}
\dL f^{\ast} \colon D^b \Coh(Y) \to D^b \Coh(X)
\end{align*}
the morphism (\ref{mor:Y})  
induces an isomorphism on the tangent space. 
Therefore (\ref{mor:Y}) is an isomorphism, and (iii) is proved. 
The proof of (i) is similar by considering the morphism 
\begin{align*}
X \to \mM^{\sigma_t}([\oO_x])
\end{align*} 
for $t<0$, sending $x \in X$ to $\oO_x$. 

In order to prove (ii), we need to show that
$Y$ corepresents the functor of families of 
objects in $M^{\sigma_0}([\oO_x])$. 
In Theorem~\ref{thm:3fold}, we will discuss 
a similar statement for 3-folds. Noting that 
$\oO_x$ for $x\in C$ and 
$\dL f^{\ast} \oO_0$ are $S$-equivalent to 
$\oO_C \oplus \oO_C(-1)[1]$, the same argument 
as in Theorem~\ref{thm:3fold} is applied. 
\end{proof}

\begin{rmk}
If $t<0$, then any skyscraper sheaf $\oO_x$ for
$x\in X$ is $\sigma_t$-stable of phase one. 
By using the proof of~\cite[Lemma~10.1]{Brs2}, 
we can show that $\sigma_t$ for $t<0$  
coincides with a stability condition constructed in~\cite{AB}. 
\end{rmk}

\subsection{Bogomolov-Gieseker inequality for $\sigma_0$-semistable objects}
The rest of this section is devoted to 
proving Proposition~\ref{prop:support}. 
The key ingredient is to prove Bogomolov-Gieseker type 
inequality for $\sigma_0$-semistable objects
in $\bB_{f^{\ast}\omega}$. 
The desired inequality is proved by a somewhat 
tricky argument: we first prove a version of 
BG inequality, evaluating
$(\ch_1(\ast) \cdot f^{\ast}\omega)^2$.
 It is not enough to 
show the support property, but 
ensures the `support property in 
one direction'. 
By this property, we are able to apply 
wall-crossing argument with respect to 
the central charges $Z_{sf^{\ast}\omega}$
for $s\in \mathbb{R}_{>0}$. 
 Then BG inequality evaluating $(\ch_1(\ast))^2$
is proved by the induction on $\ch_1(E) \cdot f^{\ast}\omega$
as in~\cite[Theorem~7.3.1]{BMT}. 

The following is the BG type inequality 
evaluating $(\ch_1(\ast) \cdot f^{\ast}\omega)^2$.
\begin{prop}\label{prop:wBGin}
For any $\sigma_0$-semistable object 
$E \in \bB_{f^{\ast}\omega}$, we have the inequality
\begin{align}\label{wBGin}
(\ch_1(E) \cdot f^{\ast}\omega)^2 \ge 2\omega^2 \ch_0(E) \ch_2(E). 
\end{align}
\end{prop}
\begin{proof}
The argument is borrowed from the proof of
the support property for the local projective plane in~\cite[Lemma~4.5]{BaMa},
although a similar inequality is not stated explicitly there. 

Let $E \in \bB_{f^{\ast}\omega}$
be a $\sigma_0$-semistable object. 
Since we may assume
 $\ch_0(E) \ch_2(E) >0$, 
we have the two possibilities: 
\begin{align*}
\ch_0(E), \ch_2(E)>0 \quad \mbox{ or } \quad \ch_0(E), \ch_2(E)<0.
\end{align*} 
We first prove the inequality (\ref{wBGin}) 
in the case $\ch_0(E), \ch_2(E)>0$.
There is an exact sequence in $\bB_{f^{\ast}\omega}$
\begin{align*}
0 \to \hH_p^{-1}(E)[1] \to E \to \hH_p^0(E) \to 0. 
\end{align*}
By the assumption, we have the two inequalities
\begin{align*}
&0<\ch_0(E) \le \ch_0(\hH_p^{0}(E)), \\ 
&0 \le \Imm Z_{f^{\ast}\omega}(\hH_p^{0}(E)) \le \Imm Z_{f^{\ast}\omega}(E).
\end{align*}
The above inequalities immediately imply 
$\mu_{f^{\ast}\omega}(E) \ge \mu_{f^{\ast}\omega}(\hH_p^{0}(E))$. 
On the other hand, there is a 
surjection $\hH_p^0(E) \twoheadrightarrow F$ in
$\PPer(X/Y)$ 
such that $F$ is $\mu_{f^{\ast}\omega}$-semistable with $\ch_0(F)>0$
and 
\begin{align}\label{ineq:mu}
0<\mu_{f^{\ast}\omega}(F) \le \mu_{f^{\ast}\omega}(\hH^{0}_p(E))
\le \mu_{f^{\ast}\omega}(E).
\end{align}
We have the composition 
\begin{align*}
E \twoheadrightarrow \hH_p^{0}(E) \to F,
\end{align*} which is 
surjective in $\bB_{f^{\ast}\omega}$.
Therefore the $\sigma_0$-stability of 
$E$ yields $\arg Z_{f^{\ast}\omega}(E) \le \arg Z_{f^{\ast}\omega}(F)$, or 
equivalently 
\begin{align*}
\frac{-\ch_2(E) + \omega^2 \ch_0(E)/2}{\ch_1(E) \cdot f^{\ast}\omega}
\ge \frac{-\ch_2(F) + \omega^2 \ch_0(F)/2}{\ch_1(F) \cdot f^{\ast}\omega}.
\end{align*}
Combined with (\ref{ineq:mu})
and the assumption $\ch_2(E)>0$, 
we obtain the inequality
\begin{align}\label{ineq:ch2}
0< \frac{\ch_2(E)}{\ch_1(E) \cdot f^{\ast}\omega} \le 
\frac{\ch_2(F)}{\ch_1(F) \cdot f^{\ast}\omega}.
\end{align} 
Then the inequality (\ref{wBGin}) follows from 
\begin{align*}
\frac{1}{2\omega^2} &\ge \frac{\ch_0(F)}{\ch_1(F) \cdot f^{\ast}\omega} \cdot 
\frac{\ch_2(F)}{\ch_1(F) \cdot f^{\ast}\omega} \\
&\ge \frac{\ch_0(E)}{\ch_1(E) \cdot f^{\ast}\omega} \cdot 
\frac{\ch_2(E)}{\ch_1(E) \cdot f^{\ast}\omega}.
\end{align*}
Here the first inequality follows from Remark~\ref{rmk:weak}, 
and the second inequality follows from (\ref{ineq:mu}) and (\ref{ineq:ch2}). 

The proof for the case of $\ch_0(E), \ch_2(E)<0$
is similar, by replacing the quotient $\hH_p^{0}(E) \twoheadrightarrow F$
by a $\mu_{f^{\ast}\omega}$-semistable
subobject $G \subset \hH_p^{-1}(E)$ with 
$\mu_{f^{\ast}\omega}(G) \ge \mu_{f^{\ast}\omega}(\hH_p^{-1}(E))$. 
\end{proof}
For $s\in \mathbb{R}_{>0}$, 
note that $\bB_{sf^{\ast}\omega}=\bB_{f^{\ast}\omega}$. 
Let us consider the element 
\begin{align*}
\gamma_s \cneq (Z_{sf^{\ast}\omega}, \bB_{f^{\ast}\omega})
\in \Stab^{\dag}(X). 
\end{align*}
Using the inequality (\ref{wBGin}), we show the following 
proposition. 
\begin{prop}\label{prop:cont}
The map $\gamma \colon s \mapsto \gamma_s$ is 
continuous. Moreover for each $E \in D^b \Coh(X)$, 
 there is a wall and chamber structure 
on the image of $\gamma$ 
so that the Harder-Narasimhan filtration of $E$
is constant on each chamber. 
\end{prop}
\begin{proof}
As for the first statement, by~\cite[Proposition~6.3, Theorem~7.1]{Brs1},
it is enough to show the following:
for any $s \in \mathbb{R}_{>0}$, 
there is a positive constant $C$ such that 
\begin{align}\label{const:C}
\frac{\lvert
Z_{s f^{\ast}\omega}(E) -Z_{f^{\ast}\omega}(E) \rvert}{\lvert Z_{f^{\ast}\omega}(E) \rvert} \le C 
\end{align}
for any $\sigma_0(=\gamma_1)$-semistable object $E \in \bB_{f^{\ast}\omega}$.
If $\ch_0(E)=0$, then the LHS of (\ref{const:C}) 
is less than or equal to $\lvert s-1 \rvert$.  
Hence we may assume that $\ch_0(E) \neq 0$, and 
set 
\begin{align*}
x= \frac{\ch_1(E) \cdot f^{\ast}\omega}{\ch_0(E)}, \quad 
y= \frac{\ch_2(E)}{\ch_0(E)}. 
\end{align*}
Then the LHS of (\ref{const:C}) is written as 
\begin{align*}
\lvert s-1 \rvert 
\sqrt{ \frac{(s+1)^2/4 +x^2}{(-y + \omega^2/2)^2 + x^2}}. 
\end{align*}
By the inequality (\ref{wBGin}), we have 
$x^2 \ge 2\omega^2 y$. 
Hence the bound (\ref{const:C}) is obtained by 
the following elementary fact: 
\begin{align*}
\mathrm{inf} \{ (-y+ \omega^2/2)^2 + x^2 : (x, y) \in \mathbb{R}^2, 
x^2 \ge 2\omega^2 y  \} >0. 
\end{align*}
The latter statement follows from the same argument 
describing the wall in~\cite[Proposition~9.3]{Brs2}, 
after replacing~\cite[Lemma~5.1]{Brs2}
by Proposition~\ref{prop:wBGin}. 
\end{proof}
The following lemma is required to show 
another version of BG type inequality.  
\begin{lem}\label{lem:ineq:per1}
There is a constant $C_{\omega} >0$, 
which depends only on the class 
$[\omega] \in \mathbb{P}(H^2(Y)_{\mathbb{R}})$, such that 
for any $E \in \PPer_{\le 1}(X/Y)$ with 
$\Hom(\PPer_{0}(X/Y), E)=0$, we have the inequality
\begin{align}\notag
\ch_1(E)^2 \omega^2  +C_{\omega}(\ch_1(E) \cdot f^{\ast}\omega)^2 \ge 0. 
\end{align}
\end{lem}
\begin{proof}
Let us write 
\begin{align*}
\ch_1(E)=f^{\ast}\alpha + k[C], 
\end{align*} for some 
$k\in \mathbb{Z}$ and $\alpha \in H^2(Y)$. 
Since $\ch_1(E) \cdot C \ge 0$
by Lemma~\ref{lem:Cch}, we have $k\le 0$. 
Next let us take $a_{\omega}>0$ so that 
$f^{\ast}\omega -a_{\omega}\sqrt{\omega^2}C$ is ample. 
Note that we can take $a_{\omega}$ so that it depends 
only on the class $[\omega] \in \mathbb{P}(H^2(Y)_{\mathbb{R}})$. 
By Proposition~\ref{prop:collect} (ii)
and the assumption, we must have $\hH^{-1}(E)=0$. 
Therefore $E$ is a one dimensional sheaf, 
hence $\ch_1(E) \cdot (f^{\ast}\omega -a_{\omega}\sqrt{\omega^2}C) \ge 0$. 
This implies 
\begin{align*}
-\frac{\alpha \cdot \omega}{a_{\omega}\sqrt{\omega^2}} \le k \le 0. 
\end{align*}
On the other hand, since $\alpha$ is an effective 
class on $Y$, there is $C_{\omega}'>0$ which 
only depends on the class $[\omega] \in \mathbb{P}(H^2(Y)_{\mathbb{R}})$
such that (cf.~the proof of~\cite[Corollary~7.3.1]{BMT})
\begin{align*}
\alpha^2 \omega^2  + (\alpha \cdot \omega)^2 C_{\omega}' \ge 0. 
\end{align*}
By combining the above inequalities, we obtain 
\begin{align*}
\ch_1(E)^2 &= \alpha^2 -k^2 \\
& \ge -\left( C_{\omega}' + \frac{1}{a_{\omega}^2}
  \right) \frac{(\alpha \cdot \omega)^2}{\omega^2}. 
\end{align*}
Hence $C_{\omega} \cneq C_{\omega}' + 1/a_{\omega}^2$
satisfies the desired property. 
\end{proof}
Let $\dD$ be the set of 
isomorphism classes of objects $E \in \bB_{f^{\ast}\omega}$
satisfying one of the following: 
\begin{itemize}
\item $E \in \PPer(X/Y)$, $\ch_0(E)>0$
 and it is $\mu_{f^{\ast}\omega}$-semistable. 
\item $E \in \PPer_{\le 1}(X/Y)$ and it 
satisfies $\Hom(\PPer_{0}(X/Y), E)=0$. 
\item $E$ fits into an exact sequence in $\bB_{f^{\ast}\omega}$, 
\begin{align}\label{FET}
0 \to F[1] \to E \to T \to 0
\end{align}
where $F \in \PPer(X/Y)$ is $\mu_{f^{\ast}\omega}$-semistable, 
$T \in \PPer_{0}(X/Y)$, and $E$ satisfies 
$\Hom(\PPer_{0}(X/Y), E)=0$. 
\end{itemize}
\begin{lem}\label{lem:sgg}
Suppose that an object $E \in \bB_{f^{\ast}\omega}$ 
satisfies $\Imm Z_{f^{\ast}\omega}(E)>0$ and 
$Z_{sf^{\ast}\omega}$-semistable for $s \gg 0$. 
Then we have $E \in \dD$. 
\end{lem}
\begin{proof}
The proof is given by the same argument as 
in~\cite[Proposition~14.2]{Brs2}, \cite[Lemma~4.2]{Bay}, 
after replacing $\Coh(X)$ by $\PPer(X/Y)$. 
We omit the detail. 
\end{proof}
Let $c \in \mathbb{Z}_{>0}$ be
\begin{align*}
c \cneq \mathrm{min} \{ \ch_1(E) \cdot f^{\ast}\omega >0 : E \in \bB_{f^{\ast}\omega}\}. 
\end{align*}
\begin{lem}\label{lem:small}
For an object $E \in \bB_{f^{\ast}\omega}$, suppose that 
$\ch_1(E) \cdot f^{\ast}\omega=c$. Then $E \in \dD$. 
\end{lem}
\begin{proof}
The same argument of~\cite[Lemma~7.2.2]{BMT}
is applied, after replacing $\Coh(X)$ by $\PPer(X/Y)$. 
Again we omit the detail. 
\end{proof}
Combining the above results, we prove the 
BG type inequality evaluating $(\ch_1(\ast))^2$.
\begin{thm}\label{thm:sBG}
There is a constant $C_{\omega}>0$, which depends only 
on the class $[\omega] \in \mathbb{P}(H^2(Y)_{\mathbb{R}})$, such that 
the following holds: 
for any $\sigma_0$-semistable object $E \in \bB_{f^{\ast}\omega}$
with $\ch_1(E) \cdot f^{\ast}\omega >0$, we have the  
inequality
\begin{align}\label{sBGin}
\ch_1(E)^2 -2\ch_0(E) \ch_2(E) 
 +C_{\omega}\frac{(\ch_1(E) \cdot f^{\ast}\omega)^2}{\omega^2} \ge 0. 
\end{align}
\end{thm}
\begin{proof}
The 
proof proceeds as in the proof of~\cite[Theorem~7.3.1]{BMT},
so we just give an outline of the proof. 
We show that the inequality (\ref{sBGin})
holds for any $Z_{sf^{\ast}\omega}$-semistable 
object $E \in \bB_{f^{\ast}\omega}$
with $\ch_1(E) \cdot f^{\ast}\omega>0$
and $s\in \mathbb{R}_{>0}$,
by the induction on $\ch_1(E) \cdot f^{\ast}\omega$. 

The first step is the case that
$\ch_1(E) \cdot f^{\ast}\omega =c$. 
In this case, 
we have $E\in \dD$ by Lemma~\ref{lem:small}. 
Hence the inequality (\ref{sBGin})
follows from Proposition~\ref{prop:BGper}, Lemma~\ref{lem:ineq:per1}, 
except the case
that $E$ fits into an exact sequence (\ref{FET}). 
In the last case, the inequality (\ref{BG:per})
can be proved along with the same argument of 
Proposition~\ref{prop:BGper}. 
Indeed it is enough to replace the sequence (\ref{TEF}) 
by (\ref{FET}) and the same computation works. 

Suppose that (\ref{sBGin})
holds for any $Z_{s'f^{\ast}\omega}$-semistable 
object $F\in \bB_{f^{\ast}\omega}$
with $s'>0$ and 
$0<\ch_1(F) \cdot f^{\ast}\omega < \ch_1(E) \cdot f^{\ast}\omega$.
We consider the following set: 
\begin{align*}
\sS \cneq  \{ s'' \ge s : 
E \mbox{ is } Z_{s''f^{\ast}\omega} \mbox{-semistable } \},
\end{align*}
and set $s_0 \cneq \mathrm{sup}\{ s'' \in \sS \}$. 
If $s_0 =\infty$, we have $E\in \dD$ by 
Lemma~\ref{lem:sgg}, so the inequality (\ref{sBGin}) holds. 
Otherwise, by the result of Proposition~\ref{prop:cont}, 
the set $\sS$ is a closed subset of $\mathbb{R}_{>0}$.
By the existence of wall and chamber structure 
in Proposition~\ref{prop:cont}, 
we have the following: 
there is $0 < \epsilon \ll 1$
and an exact sequence 
in $\bB_{f^{\ast}\omega}$
\begin{align*}
0 \to E_1 \to E \to E_2 \to 0
\end{align*}
such that 
\begin{align}\label{Z:eq}
\arg Z_{s_0 f^{\ast}\omega}(E_1) &= \arg Z_{s_0 f^{\ast}\omega}(E_2), \\
\label{Z:ineq}
\arg Z_{(s_0 + \epsilon) f^{\ast}\omega}(E_1) &>
\arg Z_{(s_0 + \epsilon) f^{\ast}\omega}(E_2).
\end{align}
Note that $E_i$ are $Z_{s_0 f^{\ast}\omega}$-semistable 
with $0<\ch_1(E_i) \cdot f^{\ast}\omega < \ch_1(E) \cdot f^{\ast}\omega$. 
Hence by the assumption of the induction, the 
objects $E_i$ satisfy the inequality (\ref{sBGin}). 
Together with the equality (\ref{Z:eq})
and the inequality (\ref{Z:ineq}), 
we can apply the exactly same computation 
in~\cite[Theorem~7.3.1]{BMT}, 
and show that $E$ satisfies (\ref{sBGin}).  
\end{proof}
\begin{cor}
Let $C_{\omega}>0$ be as in Theorem~\ref{thm:sBG}. 
Then for any $\sigma_0$-stable object $E\in \bB_{f^{\ast}\omega}$, 
we have the inequality
\begin{align}\label{sBGin2}
\ch_1(E)^2 -2\ch_0(E) \ch_2(E)
+C_{\omega}\frac{(\ch_1(E) \cdot f^{\ast}\omega)^2}{\omega^2} 
\ge -1.
\end{align}
\end{cor}
\begin{proof}
By Theorem~\ref{thm:sBG}, we may assume 
$\ch_1(E) \cdot f^{\ast}\omega=0$. 
In this case, $E$ is contained in the
RHS of (\ref{Imm=0}). 
Since $E$ is $\sigma_0$-stable, 
either $E=F[1]$ for a $\mu_{f^{\ast}\omega}$-semistable 
sheaf $F$ with $\mu_{f^{\ast}\omega}(F)=0$ or 
$E$ is a simple object in $\PPer_{0}(X/Y)$. 
In the former case, the inequality (\ref{sBGin2}) 
follows from Proposition~\ref{prop:BGper}. 
In the latter case, the inequality (\ref{sBGin2})
follows since it is satisfied for
simple objects in $\PPer_{0}(X/Y)$
described in Proposition~\ref{prop:collect} (i). 
\end{proof}
\begin{rmk}
The equality is achieved in (\ref{sBGin2})
if and 
only if $E$ is isomorphic to 
$\oO_C$ or $\oO_C(-1)[1]$. 
\end{rmk}
\subsection{Proof of Proposition~\ref{prop:support}}\label{subsec:support}
\begin{proof}
We first fix the norm $\lVert \ast \rVert$
on $N(X)_{\mathbb{R}}$. 
Let us embed $N(X)_{\mathbb{R}}$ into 
$H^{\ast}(X, \mathbb{R})$ via the Chern 
character map.
For an element $(r, \beta, n) \in H^0(X) \oplus H^2(X)
\oplus H^4(X)$, we can write 
\begin{align*}
\beta=\beta_{+} + \beta_{-}, \quad 
\beta_{+} \in \mathbb{R}[f^{\ast}\omega], \quad
\beta_{-} \in (f^{\ast}\omega)^{\perp}. 
\end{align*}
Note that $\beta_{-}^2 \le 0$ by 
the Hodge index theorem. 
We set 
\begin{align*}
\lVert (r, \beta, n) \rVert \cneq 
\mathrm{max}\{ \lvert r \rvert, 
\lvert n \rvert, 
\lvert \beta^{+} \cdot f^{\ast}\omega \rvert, 
 \sqrt{-\beta_{-}^2}  \}. 
\end{align*}
Let us bound $\lVert E \rVert / \lvert Z_{f^{\ast}\omega}(E) \rvert$
for $\sigma_0$-semistable objects $E\in \bB_{f^{\ast}\omega}$. 
By the triangle inequality, it is enough 
to bound it for $\sigma_0$-stable objects. 
If we write $\ch(E)=(r, \beta_{+} + \beta_{-}, n)$ as above, 
then $\lVert E \rVert / \lvert Z_{f^{\ast}\omega}(E) \rvert$
coincides with 
\begin{align*}
\frac{\mathrm{max}\{ \lvert r \rvert, 
\lvert n \rvert, 
\lvert \beta^{+} \cdot f^{\ast}\omega \rvert, 
 \sqrt{-\beta_{-}^2} \}}{\sqrt{\left( -n+ r\omega^2/2 \right)^2
+(\beta_{+} \cdot f^{\ast}\omega)^2}}. 
\end{align*}
By the same argument of Proposition~\ref{prop:cont},  
the values 
\begin{align*}
\frac{\rvert r \rvert}{\lvert Z_{f^{\ast}\omega}(E) \rvert}, \
\frac{\rvert n \rvert}{\lvert Z_{f^{\ast}\omega}(E) \rvert}, \ 
\frac{\rvert \beta^{+} \cdot f^{\ast}\omega\rvert}{\lvert
Z_{f^{\ast}\omega}(E) \rvert}
\end{align*}
can be shown to be bounded. 
It remains to show the boundedness of 
$\rvert \sqrt{-\beta_{-}^2}  \rvert/\lvert Z_{f^{\ast}\omega}(E) \rvert$.

In the above notation, the inequality (\ref{sBGin2})
is written as
\begin{align*}
\beta_{+}^2 + \beta_{-}^2 -2rn + 
\frac{C_{\omega}(\beta_{+} \cdot f^{\ast}\omega)^2}{\omega^2} \ge -1.
\end{align*}
Noting $\beta_{+}^2 \omega^2 = (\beta_{+} \cdot f^{\ast}\omega)^2$, 
we have
\begin{align}\label{RHS}
\frac{-\beta_{-}^2}{\lvert Z_{f^{\ast}\omega}(E) \rvert^2}
\le \frac{1+(C_{\omega}+1)\beta_{+}^2 -2rn}{\left(-n + r\omega^2/2 \right)^2 +
\beta_{+}^2 \omega^2}. 
\end{align}
In order to obtain the upper bound of the RHS, 
we may assume $rn<0$. 
By setting $u= -n/r >0$, the RHS of (\ref{RHS})
is shown 
to be bounded above by the fact
\begin{align*}
\mathrm{sup} \left\{ 
\frac{u}{\left( u+ \omega^2/2 \right)^2} : u>0  \right\} < \infty. 
\end{align*}
\end{proof}

\section{Extremal contractions of projective 3-folds}\label{sec:3fold}
In what follows, $X$ is a smooth projective 3-fold 
and $f$ is a birational extremal contraction
\begin{align}\label{extremal2}
f \colon X \to Y,
\end{align}
i.e. $Y$ is a normal projective variety, 
$-K_X$ is $f$-ample and the relative 
Picard number of $f$ is equal to one. 
In this section, we 
construct the perverse heart associated to 
$f$, and 
establish a derived equivalence between 
$X$ and a certain sheaf of non-commutative 
algebras on $Y$. 

\subsection{Classification}\label{subsec:class}
By~\cite{Mori}, the extremal contraction (\ref{extremal2})
is classified into five types: 

{\bf Type I:} $Y$ is smooth and $f$ is a blow up at a smooth 
curve. 

{\bf Type II:} $Y$ is smooth and $f$ is a blow-up 
at a point. 

{\bf Type III:} $Y$ has an ordinary double point, and 
$f$ is a blow-up at the singular point. 

{\bf Type IV:}
 $Y$ has an orbifold singularity $\mathbb{C}^3/(\mathbb{Z}/2\mathbb{Z})$
and $f$ is a blow-up at the singular point. 

{\bf Type V:} $Y$ has a $cA_2$-singularity and $f$ is a blow-up 
at the singular point. 

We denote by $D\subset X$ the exceptional divisor. 
When $f(D)$ is a point $0\in Y$, 
(i.e. $f$ is not type I,)
we denote by $\widehat{\oO}_{Y, 0}$ the completion of 
$\oO_{Y, 0}$ at $0$. 
In types III, IV and V, we have 
\begin{align*}
\widehat{\oO}_{Y, 0} \cong 
\left\{ \begin{array}{ll}
\mathbb{C}\db[x, y, z, w \db]/(xy + zw) & \mbox{ in type III } \\
\mathbb{C}\db[x, y, z\db]^{(2)} & \mbox{ in type IV } \\
\mathbb{C}\db[x, y, z, w \db]/(xy + z^2 + w^3) & \mbox{ in type V. }
\end{array}
\right.
\end{align*}
and 
\begin{align*}
(D, \oO_D(D)) \cong \left\{
\begin{array}{ll}
(\mathbb{P}^1 \times \mathbb{P}^1, 
\oO_{\mathbb{P}^1 \times \mathbb{P}^1}(-1, -1)) & \mbox{ in type III } \\
(\mathbb{P}^2, \oO_{\mathbb{P}^2}(-2)) & \mbox{ in type IV } \\
(D \subset \mathbb{P}^3, \oO_{D}(-1)) & \mbox{ in type V }
\end{array}
\right. 
\end{align*}
Here $\mathbb{C}\db[x, y, z\db]^{(2)}$
is the subring of invariants of 
$\mathbb{C}\db[x, y, z \db]$ for the involution 
$(x, y, z) \mapsto (-x, -y, -z)$, 
and $D \subset \mathbb{P}^3$ in type V
is the normal singular quadric surface.

\subsection{Perverse t-structure for extremal contractions}
\label{subsec:two}
In this subsection, we construct a
perverse t-structure following~\cite{TU}. 
We set 
\begin{align*}
\eE_0 \cneq \oO_X \oplus \oO_X(D). 
\end{align*}
First we note the following lemma: 
\begin{lem}\label{lem:note}
(i) The line bundle $\oO_X(-D)$ is $f$-ample
and $f$-globally generated. 

(ii) The vector bundles
$\eE_0, \eE_0|_{D}$ satisfy 
the following for $i>0$: 
\begin{align*}
&\hH^i ( \dR f_{\ast} \hH om(\eE_0, \eE_0) ) \cong 0, \\
&\hH^i (\dR f|_{D \ast} \hH om(\eE_0|_D, \eE_0|_D)) \cong 0. 
\end{align*}
\end{lem}
\begin{proof}
Both of (i) and (ii) are easily proved 
by the classification in Subsection~\ref{subsec:class}, 
or using vanishing theorem. 
\end{proof}
The result of (ii) in particular implies 
\begin{align*}
R^2 f_{\ast} \oO_X(D)=R^2 f|_{D\ast}\oO_D(D)=0. 
\end{align*}
Hence we are in the situation 
where the result of~\cite{TU} is applied for 
$f$ and $f|_{D}$.
In particular, we can define the perverse t-structures 
on $D^b \Coh(X)$ and $D^b \Coh(D)$. 
Below we recall their constructions. 

Let $\mathrsfs{A}_0 \cneq f_{\ast} \eE nd(\eE_0)$
be the sheaf of non-commutative algebras on $Y$, 
and $\Phi_0$ be the functor 
\begin{align*}
\Phi_0 \cneq \dR f_{\ast} \dR \hH om(\eE_0, \ast) 
\colon D^b \Coh(X) \to D^b \Coh(\mathrsfs{A}_0).
\end{align*}
The above functor is 
an equivalence when $f$ is type I by~\cite{MVB}, 
since the dimension of the fiber of $f$ is at most one 
dimensional in this case. 
However $\Phi_0$ is not an equivalence in other cases. Indeed, 
the following subcategory is non-trivial:
\begin{align*}
\cC \cneq \{ E\in D^b \Coh(X) : \Phi_0(E) \cong 0\}.
\end{align*}
The key result of~\cite[Theorem~6.1]{TU} is that 
the standard t-structure on $D^b \Coh(X)$ induces a 
t-structure on $\cC$. 
Let 
\begin{align*}
(\cC^{\le 0}, \cC^{\ge 0})
\end{align*} be the induced 
t-structure with heart $\cC^0 \cneq \cC \cap \Coh(X)$. 
The heart $\ppPPer(X/Y)$
\footnote{This was denoted by $\ppPPer(X/\mathrsfs{A}_0)$ in~\cite{TU}.} 
is defined as follows: 
\begin{defi}
We define $\ppPPer(X/Y)$ to be
\begin{align*}
\ppPPer(X/Y) 
\cneq \left\{ 
E \in D^b \Coh(X) : 
\begin{array}{c}
\Phi_0(E) \in \Coh(\mathrsfs{A}_0), \\
\Hom(\cC^{<p}, E)=\Hom(E, \cC^{>p})=0 
\end{array}
\right\}. 
\end{align*}
\end{defi}
As in~\cite{TU}, the category 
$\ppPPer(X/Y)$ is the heart of a bounded t-structure 
on $D^b \Coh(X)$. 
In the 3-fold case, we use 
the perversity $p=0$\footnote{Here we 
use the different perversity from the 
surface case.}.
By abuse of notation, we set
\begin{align*}
\PPer(X/Y) \cneq \oPPer(X/Y). 
\end{align*}
Again we denote by 
$\hH_p^{i}(\ast)$ the $i$-th cohomology 
functor
with respect to the t-structure with 
heart $\PPer(X/Y)$. 

By replacing $\eE_0$, $\mathrsfs{A}_0$
 by $\eE_{0}|_{D}$, $\mathrsfs{A}_{D, 0} \cneq f|_{D\ast}\eE nd(\eE|_{D})$
respectively,   
we can define the similar 
functor
\begin{align*}
\Phi_{D, 0}=\dR f|_{D\ast} \dR \hH om(\eE|_{D}, \ast)
 \colon 
D^b \Coh(D) \to D^b \Coh(\mathrsfs{A}_{D, 0})
\end{align*}
and a similar
subcategory
\begin{align*}
\cC_{D} \cneq \{ E \in D^b \Coh(D) : 
\Phi_{D, 0}(E) \cong 0 \},
\end{align*}
such that $\cC_{D}^{0} \cneq \cC_{D} \cap \Coh(D)$ 
is the heart of a bounded t-structure on $\cC_{D}$. 
The perverse heart is also defined similarly, 
\begin{align*}
\PPer(D) 
\cneq \left\{ 
E \in D^b \Coh(D) : 
\begin{array}{c}
\Phi_{D, 0}(E) \in \Coh(\mathrsfs{A}_{D, 0}), \\
\Hom(\cC^{<0}_{D}, E)=\Hom(E, \cC^{>0}_{D})=0 
\end{array}
\right\}. 
\end{align*}
\begin{rmk}
When $f(D)$ is a point, 
then $\mathrsfs{A}_{D, 0}$ is the path 
algebra of a quiver
with two vertices and $\dim H^0(D, \oO_D(-D))$-arrows
between them in one direction.  
\end{rmk}

\subsection{Derived equivalence}
The purpose here is to 
give an analogue of Theorem~\ref{thm:nc}
for our threefold situation. 
Suppose that $f(D)$ is a point 
$0=f(D) \in Y$. 
If we denote by 
$m_0 \subset \oO_Y$ the sheaf of defining 
ideal of $0 \in Y$, 
then the sheaf of algebras $\mathrsfs{A}_0$
is written 
as a matrix
\begin{align}\label{A:matrix}
\mathrsfs{A}_0= \left( \begin{array}{cc}
\oO_Y & m_0  \\
\oO_Y & \oO_Y
\end{array}\right),
\end{align}
since $f_{\ast}\oO_X(D) \cong \oO_Y$
and $f_{\ast}\oO_X(-D)=m_0$. 
Let $S \in \Coh(\mathrsfs{A}_0)$ be the object, 
which is isomorphic to $\oO_Y/m_0$ as a 
$\oO_Y$-module, and 
the right $\mathrsfs{A}_0$-action is given by 
\begin{align*}
x \cdot \left(\begin{array}{cc}
a & b \\
c & d
\end{array}  \right) =x \cdot a. 
\end{align*}
We have the following lemma: 
\begin{lem}\label{lem:OOS}
Suppose that $f(D)$ is a point. Then 
there is a distinguished triangle in $D^b \Coh(\mathrsfs{A}_0)$,
\begin{align}\label{dist:S}
\Phi_0(\oO_X(D)) \to \Phi_0(\oO_X(2D)) \to S^{\oplus l}[-2]. 
\end{align}
Here $l \cneq \dim H^2(D, \oO_D(2D))$. 
\end{lem}
\begin{proof}
By the exact sequence of sheaves
\begin{align*}
0 \to \oO_X(D) \to \oO_X(2D) \to \oO_D(2D) \to 0.
\end{align*}
we have the distinguished triangle
\begin{align*}
\Phi_0(\oO_X(D)) \to \Phi_0(\oO_X(2D)) \to \Phi_0(\oO_D(2D)). 
\end{align*}
On the other hand, we have 
the isomorphism $\Phi_0(\oO_D(2D)) \cong S^{\oplus l}[-2]$. 
\end{proof}
If $f$ is either type III, IV or V, 
we have $H^2(D, \oO_D(2D)) \neq 0$. In this case, 
the above lemma implies that the vector bundle
\begin{align*}
\oO_X \oplus \oO_X(D) \oplus \oO_X(2D)
\end{align*}
does not satisfy the property similar to 
Lemma~\ref{lem:note} (ii). 
In order to give a derived equivalence 
as in Theorem~\ref{thm:intro0}, 
a modification of $\oO_X(2D)$ 
is required.  
By improving the argument of~\cite{TU}, we have 
the following result: 
\begin{thm}\label{thm:tilting}
Let $X$ be a smooth projective 3-fold,
$f \colon X \to Y$ an extremal contraction
and $D\subset X$ the exceptional divisor. 
Then there is a vector bundle $\vV$ on 
$X$ satisfying 
\begin{itemize}
\item If $f(D)$ is a curve, then $\vV =0$. 
\item If $f(D)$ is a point, then $\vV$ fits into 
an exact sequence of sheaves
\begin{align}\label{desire}
0 \to \oO_X(2D) \to
 \vV \to \oO_X(D-2f^{\ast}H)^{\oplus el}  
\to \oO_X(-f^{\ast}H)^{\oplus l} \to 0,
\end{align}
where $H$ is a sufficiently ample divisor on $Y$
and $e$, $l$ are given by 
\begin{align*}
e= \dim H^0(X, \oO_X(f^{\ast}H-D)), \quad 
l=\dim H^2(D, \oO_D(2D)), 
\end{align*}
\end{itemize}  
 such that, 
setting $\eE=\oO_X \oplus \oO_X(D) \oplus \vV$ and
 $\mathrsfs{A}=f_{\ast}\eE nd(\eE)$, 
we have the equivalence of derived categories
\begin{align}\label{ext:der}
\Phi \cneq
\dR f_{\ast} \dR \hH om(\eE, \ast) \colon 
D^b \Coh(X) \stackrel{\sim}{\to} D^b \Coh(\mathrsfs{A}),
\end{align}
which restricts to an equivalence between 
$\PPer(X/Y)$ and $\Coh(\mathrsfs{A})$. 
\end{thm}
\begin{proof}
If $f(D)$ is a curve, the result follows from~\cite{MVB}. 
Below we assume $f(D)$ is a point $0\in Y$. 

Let $H$ be a sufficiently ample divisor
on $Y$ so that $m_0(H)$ is globally generated, 
and satisfies
\begin{align}\label{Ext1}
\Ext^{j}_{\mathrsfs{A}_0}(\mathrsfs{A}_0, \oO_Y(iH) \otimes_{\oO_Y} 
\mathrsfs{A}_0) 
&\cong 
H^{j}(Y, \oO_Y(iH) \otimes_{\oO_Y} \mathrsfs{A}_0) \\
\notag
&\cong 0
\end{align}
for any $i\ge 1$ and $j\ge 1$. 
Let us take a basis $x_1, \cdots, x_e$, 
\begin{align}\label{ext:sec}
x_i \in \Gamma(X, m_0(H)) = \Gamma(X, \oO_X(f^{\ast}H-D)). 
\end{align}
We have the right exact sequence in $\Coh(\mathrsfs{A}_0)$
\begin{align}\label{right}
\mathrsfs{A}_0 \oplus 
(\oO_Y(-H) \otimes_{\oO_Y}\mathrsfs{A}_0^{\oplus e})
\to \mathrsfs{A}_0 \to S \to 0.
\end{align}
Here the middle morphism takes $1 \in \mathrsfs{A}_0$ to $1 \in S$, 
and the left one
is given by the $(e+1)$-matrices:
\begin{align}\notag
\left( \begin{array}{cc}
0 & 0 \\
0 & 1
\end{array}  \right), \quad 
\left( \begin{array}{cc}
0 & x_i \\
0 & 0
\end{array} \right), \quad 1\le i\le e,
\end{align}
in the expression (\ref{A:matrix}).
Below if we write $(\fF \to \gG)$, this always 
means the two term complex with $\fF$ located in 
degree one. 
By taking the $l$-direct sum
of the resolution (\ref{right}) tensored by 
$\oO_Y(-H)$, 
and combining the distinguished triangle (\ref{dist:S}),
we obtain the morphisms 
\begin{align}\notag
\oO_Y(-H) \otimes_{\oO_Y}
\left(
\mathrsfs{A}_0 \oplus  
(\oO_Y(-H) \otimes_{\oO_Y}\mathrsfs{A}_0^{\oplus e}) 
\to \mathrsfs{A}_0 \right)^{\oplus l} \\ 
\notag
\stackrel{\psi}{\to} S^{\oplus l}[-2] \to \Phi_0(\oO_X(D))[1].
\end{align}
Since $\Phi_0(\oO_X(D))$ is a direct summand of $\mathrsfs{A}_0$, 
the condition (\ref{Ext1}) 
implies that the 
composition of the above morphisms is zero. 
Hence $\psi$ lifts to a
morphism  
\begin{align}\label{rest}
\oO_Y(-H) \otimes_{\oO_Y}
\left(
\mathrsfs{A}_0 \oplus  
(\oO_Y(-H) \otimes_{\oO_Y}\mathrsfs{A}_0^{\oplus e}) 
\to \mathrsfs{A}_0 \right)^{\oplus l} \to \Phi_0(\oO_X(2D)).
\end{align}
Note that the functor 
 $\Psi_0 \colon D^b \Coh(\mathrsfs{A}_0) \to D^{-}\Coh(X)$
given by  
\begin{align*}
\Psi_0(\mM) \cneq \mM \dotimes_{\mathrsfs{A}_0} \eE_0
\end{align*}
is a left adjoint of $\Phi_0$. 
Taking the adjunction of (\ref{rest}), we obtain the morphism 
\begin{align}\label{adj}
\left(
\eE_0 \oplus 
\eE_0^{\oplus e}(-f^{\ast}H)
\stackrel{\phi}{\to} \eE_{0} \right)^{\oplus l}(-f^{\ast}H) \to \oO_X(2D).
\end{align}
Since $m_0(H)$ is globally generated and 
$-D$ if $f$-globally generated, 
the morphism $\phi$ is surjective and $\Ker(\phi)$ is 
given by 
\begin{align*}
\oO_X\oplus \oO_X(-f^{\ast}H)^{\oplus e} 
\oplus \Ker(\oO_X(D-f^{\ast}H)^{\oplus e} 
\stackrel{\psi'}{\twoheadrightarrow} \oO_X).
\end{align*}
Here the morphism $\psi'$ is given by the sections (\ref{ext:sec}).  
Let $\vV'$ be the cone of the morphism (\ref{adj}). 
It fits into the exact sequence of sheaves
\begin{align*}
0 \to \oO_X(2D) \to \vV' \to \Ker(\phi)^{\oplus l}(-f^{\ast}H) \to 0. 
\end{align*}
If we take $H$ sufficiently ample, 
the vector bundle 
$\vV'$ splits into 
\begin{align*}
\vV' \cong \vV \oplus \oO_X(-f^{\ast}H)^{\oplus l}
 \oplus \oO_X(-2f^{\ast}H)^{\oplus el},
\end{align*}
where $\vV$ is a vector bundle on $X$ which fits into a desired 
exact sequence (\ref{desire}). 

We claim that $\eE=\eE_0 \oplus \vV$
gives the desired derived equivalence (\ref{ext:der}).
Let $0 \in U\subset Y$ be an
affine open neighborhood. 
Then by the construction and Lemma~\ref{lem:OOS}, we can take a 
free left $\mathrsfs{A}_{0}|_{U}$-resolution 
$\pP^{\bullet}$ of $\Phi_0(\oO_X(2D))|_{U}$,
so that its stupid filtration 
$\sigma_{\ge 1}\pP^{\bullet}$ together with the
canonical morphism
\begin{align*}
\sigma_{\ge 1}\pP^{\bullet} \to \pP^{\bullet} \stackrel{\sim}{\to}
 \Phi_0(\oO_X(2D))|_{U}
\end{align*}
is identified with (\ref{rest}) restricted to $U$. 
Hence the vector bundle $\vV'|_{f^{-1}(U)}$ fits into 
the distinguished triangle 
\begin{align*}
\Psi_{0}(\sigma_{\ge 1}\pP^{\bullet}) \to \oO_{X}(2D)|_{f^{-1}(U)} \to 
\vV'|_{f^{-1}(U)}. 
\end{align*}
By~\cite[Equation~(13)]{TU}, this implies that the vector bundle 
\begin{align*}
\eE_0|_{f^{-1}(U)} \oplus \vV'|_{f^{-1}(U)}
\end{align*}
on $f^{-1}(U)$ is nothing but a 
projective generator of $\PPer(f^{-1}(U)/U)$
constructed in~\cite[Subsection~4.4]{TU}. 
Since $\vV'|_{f^{-1}(U)}$ is a direct sum of 
$\vV|_{f^{-1}(U)}$ and $\oO_{f^{-1}(U)}^{\oplus el+l}$, 
the vector bundle $\eE|_{f^{-1}(U)}$ is also a
projective generator of $\PPer(f^{-1}(U)/U)$.  
Noting that $f$ is an isomorphism outside $0\in Y$, 
this implies that $\eE$ is a local projective generator 
of $\PPer(X/Y)$ in the sense of~\cite[Proposition~3.3.1]{MVB}.
Therefore the equivalence (\ref{ext:der}) follows 
from~\cite[Proposition~3.3.1]{MVB}.
\end{proof}

Similarly to (\ref{per:inc}), we define the 
following subcategory
\begin{align}\label{per:inc2}
\PPer_{\le i}(X/Y) \cneq \{ E \in \PPer(X/Y) : 
\Phi(E) \in \Coh_{\le i}(\mathrsfs{A})\},
\end{align}
and write $\PPer_{0}(X/Y) \cneq \PPer_{\le 0}(X/Y)$. 
\begin{rmk}\label{rmk:support}
By the construction of $\eE$,
it is easy to see that $E\in \PPer(X/Y)$
is an object in $\PPer_{0}(X/Y)$ iff
$\Supp(E)$ is contained in a finite number of 
fibers of $f$, 
 an object in $\PPer_{\le 1}(X/Y)$
iff $\Supp(E)$ is at most one dimensional 
outside $D$, and an object in $\PPer_{\le 2}(X/Y)$
iff $\Supp(E)$ is at most two dimensional. 
\end{rmk}

If $f(D)$ is a point $0\in Y$,  
we set
\begin{align*}
\widehat{Y} \cneq \Spec \widehat{\oO}_{Y, 0}, \ 
\widehat{X} \cneq X\times_{Y} \widehat{Y}. 
\end{align*}
The categories $\PPer(\widehat{X}/\widehat{Y})$, 
$\PPer_{0}(\widehat{X}/\widehat{Y})$ are 
similarly defined. 
The following
local version is also proved in a similar way. 
\begin{cor}\label{cor:cor1}
Suppose that $f(D)$ is a point $0 \in Y$. 
Then there is a vector 
bundle $\widehat{\vV}$ on $\widehat{X}$ which fits into 
an exact sequence of sheaves
\begin{align}\label{desire2}
0 \to \oO_{\widehat{X}}(2D) \to \widehat{\vV} 
\to \oO_{\widehat{X}}(D)^{\oplus kl} \to \oO_{\widehat{X}}^{\oplus l} \to 0
\end{align}
where $k$ is given by 
$k\cneq \dim H^0(D, \oO_D(-D))$
such that, setting $\widehat{\eE} =\oO_{\widehat{X}} \oplus 
\oO_{\widehat{X}}(D) \oplus \widehat{\vV}$ and 
$\widehat{A} = \End(\widehat{\eE})$, 
we have the equivalence of derived categories
\begin{align*}
\widehat{\Phi} \cneq 
\dR \Hom(\widehat{\eE}, \ast) \colon 
D^b \Coh(\widehat{X}) \stackrel{\sim}{\to}
D^b \modu (\widehat{A})
\end{align*}
which restricts to an equivalence between 
$\PPer(\widehat{X}/\widehat{Y})$ and $\modu (\widehat{A})$. 
\end{cor}
\begin{proof}
We note that $k=\dim m_0/m_0^2$. 
The proof is obtained by modifying the 
proof of Theorem~\ref{thm:tilting}
by replacing the sections (\ref{ext:sec}) 
by 
the minimal generators of $m_0 \subset \widehat{\oO}_{Y, 0}$. 
\end{proof}
We also have the version for the exceptional locus. 

\begin{cor}\label{cor:cor2}
In the situation of Corollary~\ref{cor:cor1},
setting
$\vV_{D} \cneq \widehat{\vV}|_{D}$,
$\eE_D \cneq \oO_D \oplus \oO_D(D) \oplus \vV_{D}$
and $A_D \cneq \End(\eE_D)$, 
we have the derived equivalence 
\begin{align*}
\Phi_{D} \cneq 
\dR \Hom(\eE_D, \ast) \colon D^b \Coh(D) \stackrel{\sim}{\to}
D^b \modu (A_D)
\end{align*}
which restricts to an equivalence between $\PPer(D)$
and $\modu (A_D)$.
Furthermore $\vV_{D}$
fits into the universal extension
\begin{align}\label{give:ext}
0 \to \oO_D(2D) \to \vV_D \to 
\Omega_{\mathbb{P}^{k-1}}^{\oplus l}|_{D} \to 0. 
\end{align}
Here $D$ is embedded into $\mathbb{P}^{k-1}$ by the 
linear system $| \oO_D(-D) |$.
\end{cor}
\begin{proof}
By the construction of $\widehat{\eE}$ in
Corollary~\ref{cor:cor1},
the vector bundle $\eE_{D}$ coincides with the 
construction of the tilting generator on $D^b \Coh(D)$
given in~\cite{TU}. The exact 
sequence (\ref{give:ext}) is obtained by restricting the sequence 
(\ref{desire2})
to $D$. 
Noting that 
\begin{align*}
\Ext_{D}^1(\Omega_{\mathbb{P}^{k-1}}|_{D}, \oO_D(2D)) 
\cong H^2(D, \oO_D(2D)),
\end{align*}
the universality of (\ref{give:ext})
obviously follows from the 
construction of the exact sequence (\ref{desire2}).  
\end{proof}
\begin{rmk}
The numbers $(k, l)$ are 
$(3, 0)$ in type II, $(4, 1)$ in types III, V
and $(6, 3)$ in type IV.
\end{rmk}

\begin{rmk}\label{rmk:satu}
The result of Theorem~\ref{thm:tilting}
may be applied 
to the study of minimal saturated triangulated
category of 
singular varieties. Suppose that $f$ is a
type V extremal contraction. In~\cite{KawDer}, Kawamata
proposed that the subcategory
\begin{align*}
\dD_{Y} \cneq \{ E\in D^b \Coh(X) : \dR \Hom(E, \oO_D(D))=0\}
\end{align*}
is the
minimal\footnote{Kawamata informed the author that
he later proved the minimality of $\dD_Y$.} saturated triangulated
subcategory of $D^b \Coh(X)$
which contains $f^{\ast} \mathrm{Perf}(Y)$.
By our construction of the vector bundle $\vV$, 
we have 
\begin{align*}
\oO_X \oplus \vV(-D) \in \dD_{Y},
\end{align*}
 and it is a local tilting generator of $\dD_Y$
in the sense of~\cite[Proposition~3.3.1]{MVB}. 
Hence we have the equivalence
\begin{align*}
\dR f_{\ast}
\dR \hH om(\oO_X \oplus \vV(-D), \ast) \colon 
\dD_{Y} \stackrel{\sim}{\to} D^b \Coh (\mathrsfs{B})
\end{align*}
where $\mathrsfs{B} \cneq f_{\ast} \eE nd(\oO_X \oplus \vV(-D))$. 
It would be interesting to study the category $\dD_Y$
in terms of the sheaf of non-commutative algebras $\mathrsfs{B}$. 
\end{rmk}

\subsection{Derived category of the exceptional locus}
The rest of this section is devoted to 
giving a complete 
description of simple objects in $\PPer_{0}(X/Y)$
via Theorem~\ref{thm:tilting}. 
The strategy is as follows:
in this subsection, we 
describe simple objects in $\PPer(D)$
using
Corollary~\ref{cor:cor2}.
Then in the next subsection, we show that simple objects 
in $\PPer_{0}(X/Y)$ are 
either $\oO_x$ for $x\notin D$ or 
a pushforward of some simple object in $\PPer(D)$. 

Below we use the notation in Corollary~\ref{cor:cor2}. 
First we note the following:
\begin{lem}\label{ortho}
We have $\vV_{D} \in \cC_{D}^{0}$, i.e. 
\begin{align*}
\dR \Hom_{D}(\oO_D, \vV_D)= \dR \Hom_{D}(\oO_D(D), \vV_D)=0. 
\end{align*}
\end{lem}
\begin{proof}
The result obviously follows from the exact 
sequence (\ref{give:ext}). 
\end{proof}
\begin{rmk}
A similar result is not true for 
the vector bundle $\vV$ on $X$. 
Indeed we have $\Hom(\oO_X, \vV) \neq 0$. 
\end{rmk}
Next we investigate the structure of $\vV_D$. 
\begin{lem}\label{V:orth}
Suppose that $f(D)$ is a point. Then the vector bundle
$\vV_D$ on $D$ is described as follows:
\begin{itemize}
\item If $f$ is type II, then $\vV_D =\oO_{\mathbb{P}^2}(-2)$. 
\item 
If $f$ is type III, then we have 
\begin{align*}
\vV_D \cong \oO_{\mathbb{P}^1 \times \mathbb{P}^1}(-1, -2)^{\oplus 2}
\oplus \oO_{\mathbb{P}^1 \times \mathbb{P}^1}(-2, -1)^{\oplus 2}. 
\end{align*}
\item If $f$ is type IV, then $\vV_{D} \cong \Omega_{\mathbb{P}^2}^{\oplus 8}(-1)$. 
\item 
If $f$ is type V, then $\vV_{D} \cong \uU(-1)^{\oplus 2}$, where 
$\uU$ is the non-split extension
\begin{align}\label{non-split}
0 \to \oO_D(-C) \to \uU \to \oO_D(-C) \to 0
\end{align}
and $C \subset D$ is a ruling. 
\end{itemize}
\end{lem}
\begin{proof}
The statement is obvious for type II.
As for types III and IV, the statement 
easily follows from Lemma~\ref{ortho} and using 
exceptional collections on $\mathbb{P}^2$ and 
$\mathbb{P}^1 \times \mathbb{P}^1$. 
We focus on the case of type V, which is 
a less obvious case. It may be possible to construct
an isomorphism $\uU(-1)^{\oplus 2} \cong \vV_D$ directly, but 
here
we give an 
indirect proof using the abelian 
category $\cC_{D}^{0}$ given in Subsection~\ref{subsec:two}. 

Let $D$ be the singular 
quadric defined by 
\begin{align*}
D=\{ xy+z^2=0 \} \subset \mathbb{P}^3
\end{align*}
where $[x:y:z:w]$ is the homogeneous coordinate of 
$\mathbb{P}^3$. Let $C\subset D$ be a ruling, 
defined by 
\begin{align*}
C=\{y=z=0\} \subset D. 
\end{align*}
It is easy to check that 
\begin{align}\label{Ext:D}
\Ext_{D}^i(\oO_D(-C), \oO_D(-C))=\mathbb{C}, \quad i\ge 0, 
\end{align}
using the resolution
\begin{align*}
\cdots \to \oO_D(-2)^{\oplus 2} \to \oO_D(-1)^{\oplus 2}
 \stackrel{\left(y, z \right)}{\lr} \oO_D(-C) \to 0. 
\end{align*}
Here the morphisms between rank two vector bundles
 are given by the matrix
\begin{align*}
\left(\begin{array}{cc}
-z & x \\
y & z
\end{array}   \right). 
\end{align*} 
It follows that there is a unique extension (\ref{non-split}) 
up to isomorphism, and $\uU$ is a vector bundle of rank two on $D$. 
Moreover, using (\ref{non-split}) 
and (\ref{Ext:D}), we can check that 
\begin{align}\label{ext:OU}
\Ext_{D}^i(\oO_D(-C), \uU)=0, \quad i \ge 1. 
\end{align}

Next we consider the abelian category $\cC_{D}^{0}$. 
We first note that any non-zero object $E \in \cC_{D}^{0}$ is 
a torsion free sheaf. Indeed
$E\in \cC_{D}^{0}$ implies that 
\begin{align*}
H^0(D, E)=H^0(D, E|_{H}(H))=0,
\end{align*}
where 
$H \in \lvert\oO_D(1) \rvert$ is a general element.
The above vanishing shows
 that $E$ is torsion free. 
By this fact, it immediately follows that the object
\begin{align*}
\oO_D(-3C) \in \cC_{D}^{0}
\end{align*}
 is a simple object in $\cC_{D}^{0}$. 
Conversely it is easy to check that 
an object $E \in \cC_{D}^{0}$
with $\rank(E)=1$ should be isomorphic to $\oO_D(-3C)$. 

By computing $\Hom$-groups using exact sequences (\ref{give:ext}), 
(\ref{non-split}),  
it is easy to see that 
\begin{align*}
\Hom(\uU(-1), \vV_{D}) \cong \mathbb{C}^4, \ \Hom(\oO_D(-3C), \vV_{D})
\cong \mathbb{C}^2.
\end{align*}
This implies that there is a non-trivial morphism 
\begin{align}\label{UV}
\uU(-1) \to \vV_{D}
\end{align}
such that the composition 
\begin{align*}
\oO_D(-3C) \hookrightarrow \uU(-1) \to \vV_{D}
\end{align*}
is non-zero. 
Also note that $\vV_{D} \in \cC_{D}^{0}$
by Lemma~\ref{ortho}. 
Since $\oO_D(-3C)$ is 
the unique simple objects of $\cC_{D}^0$ with rank one, 
and the sequence (\ref{non-split}) is non-split, 
the morphism (\ref{UV}) should be injective. 

Let us take the exact sequence in $\cC_{D}^{0}$
\begin{align}\label{UVF}
0 \to \uU(-1) \to \vV_{D} \to F \to 0. 
\end{align}
By the above sequence and using (\ref{ext:OU}),
we see that
\begin{align*}
\Hom(\uU(-1), F) \cong \mathbb{C}^2, \ 
\Hom(\oO_D(-3C), F) \cong \mathbb{C}. 
\end{align*} 
By the same argument as above, 
we have an injection 
$\uU(-1) \hookrightarrow F$, 
which must be an isomorphism 
since $\cC_{D}^{0}$ does not contain 
non-zero 
torsion sheaves. 
Therefore $\uU(-1) \cong F$,
and $\vV_{D} \cong \uU(-1)^{\oplus 2}$ follows
since (\ref{UVF}) splits by (\ref{ext:OU}). 
\end{proof}

Using the above lemma, we 
determine the simple objects in $\PPer(D)$. 
\begin{prop}\label{prop:PDsim}
Suppose that $f(D)$ is a point. Then the 
abelian category $\PPer(D)$ is described 
by simple objects as follows: 
\begin{itemize}
\item If $f$ is type II, we have 
\begin{align*}
\PPer(D)= \langle \oO_{\mathbb{P}^2}(-3)[2], \Omega_{\mathbb{P}^2}(-1)[1], 
\oO_{\mathbb{P}^2}(-2) \rangle. 
\end{align*}
\item If $f$ is type III, we have 
\begin{align*}
\PPer(D)= \langle \oO(-2, -2)[2], S_3(-1, -1)[1], \oO(-1, -2), \oO(-2, -1) 
\rangle. 
\end{align*}
Here $S_3$ is the kernel of the universal morphism 
\begin{align*}
0 \to S_3 \to \oO(-1, 0)^{\oplus 2} \oplus \oO(0, -1)^{\oplus 2} \to \oO \to 0. \end{align*}
\item If $f$ is type IV, we have 
\begin{align*}
\PPer(D)= \langle \oO_{\mathbb{P}^2}(-3)[2], 
S_4(-1)[1], \Omega_{\mathbb{P}^2}(-1) \rangle. 
\end{align*}
Here $S_4$ is the kernel of the universal morphism
\begin{align*}
0 \to S_4 \to \Omega_{\mathbb{P}^2}^{\oplus 3} 
\to \oO_{\mathbb{P}^2}(-1) \to 0. 
\end{align*}
\item If $f$ is type V, we have
\begin{align*}
\PPer(D)= \langle \oO_D(-2)[2], S_5(-1)[1], \oO_D(-3C) \rangle. 
\end{align*}
Here $C \subset D$ is a ruling and $S_5$ is a rank 
three vector bundle on $D$ which fits into an 
exact sequence
\begin{align}\label{SUO}
0 \to S_5 \to \uU^{\oplus 2} \to \oO_D \to 0,
\end{align}
where $\uU$ is defined by (\ref{non-split}). 
\end{itemize}
\end{prop}
\begin{proof}
Let $\vV_{D}^{\dag}$ be
\begin{align*}
\vV_{D}^{\dag} = \left\{ \begin{array}{ll}
\oO_{\mathbb{P}^2}(-2) & \mbox{in Type II}, \\
\oO_{}(-1, -2) 
\oplus \oO_{}(-2, -1) & 
\mbox{in Type III}, \\
\Omega_{\mathbb{P}^2}(-1) & \mbox{in Type IV}, \\
\uU(-1) & \mbox{in Type V}. 
\end{array}
\right. 
\end{align*}
By Corollary~\ref{cor:cor2} and
Lemma~\ref{V:orth},
setting $\eE_{D}^{\dag}=\oO_D \oplus \oO_D(-1) \oplus \vV_D^{\dag}$
and $A_{D}^{\dag}=\End(\eE_{D}^{\dag})$,  
we have the derived equivalence
\begin{align*}
\Phi_{D}^{\dag} = 
\dR \Hom(\eE_{D}^{\dag}, \ast) \colon 
D^b \Coh(D) \stackrel{\sim}{\to} D^b \modu (A_D^{\dag}),
\end{align*}
which restricts to an equivalence between 
$\PPer(D)$ and $\modu (A_{D}^{\dag})$. 
The algebra $A_{D}^{\dag}$ is a finite dimensional 
$\mathbb{C}$-algebra, and there is a finite 
number of simple objects in $\modu (A_{D}^{\dag})$. 
Therefore it is enough to find objects in 
$\PPer(D)$ which are sent to simple objects in $\modu (A_{D}^{\dag})$
after applying $\Phi_{D}^{\dag}$. 

We only discuss the type V case. 
The other cases are similarly discussed. 
By the sequence (\ref{non-split}), 
we have an isomorphism of $\mathbb{C}$-algebras
\begin{align*}
\End(\uU(-1)) \cong \mathbb{C}[T]/T^2. 
\end{align*}
Also noting Lemma~\ref{ortho}, we see that
the algebra $A_{D}^{\dag}$ is a path algebra of a quiver 
of the form
\begin{align}\label{quiver}
\xymatrix@1{
\bullet 
\ar@/^/[r] \ar@/^2pt/[r]
\ar@/_/[r] \ar@/_2pt/[r]
& \bullet 
\ar@/^/[r] \ar@/^2pt/[r]
\ar@/_/[r] \ar@/_2pt/[r]
& \bullet \ar@(dl,dr)_{T}
}
\end{align}
with relations $T^2=0$ and others which we omit. 
Therefore there are only three 
simple objects in 
$\modu (A_{D}^{\dag})$ corresponding to 
the three vertices. 

We have 
\begin{align*}
\dim_{\mathbb{C}} \Phi_{D}^{\dag}(\oO_D(-2)[2])
= \dim_{\mathbb{C}}\Phi_{D}^{\dag}(\oO_D(-3C))=1.
\end{align*}
Hence $\oO_D(-2)[2]$ and $\oO_D(-3C)$ 
are simple objects in $\PPer(D)$, and they 
correspond to the left vertex, right vertex 
of the quiver (\ref{quiver}) respectively. 
In order to describe the simple object 
corresponding to the middle vertex, 
we note that
\begin{align*}
&\dR \Hom_D(\oO_D, \oO_D(-1))=0, \\ 
&\dR \Hom_D(\oO_D(-1), \oO_D(-1))=\mathbb{C}, \\ 
&\dR \Hom_D(\uU(-1), \oO_D(-1))= \mathbb{C}^4. 
\end{align*}
In particular, we have $\oO_D(-1) \in \PPer(D)$. 

The $T$-action on $\Hom_D(\uU(-1), \oO_D(-1))$
is given by the composition
\begin{align*}
\Hom_D(\uU(-1), \oO_D(-1)) \twoheadrightarrow 
&\Hom_{D}(\oO_D(-3C), \oO_D(-1)) \\
& \cong \mathbb{C}^{\oplus 2}
\hookrightarrow\Hom_D(\uU(-1), \oO_D(-1)),
\end{align*}
induced by the exact sequence (\ref{non-split}). 
Therefore we have an isomorphism as 
$\End(\uU(-1))$-modules
\begin{align*}
\Hom_{D}(\uU(-1), \oO_D(-1)) 
\cong \left( \mathbb{C}[T]/T^2 \right)^{\oplus 2},
\end{align*}
which is 
isomorphic to $\Phi_{D}^{\dag}(\uU(-1)^{\oplus 2})$ 
as $A_D^{\dag}$-modules. 
If $S_5(-1)[1] \in \PPer(D)$ is the simple object 
corresponding to the middle vertex, 
the above argument implies the 
existence of an exact sequence in $\PPer(D)$
\begin{align*}
\uU(-1)^{\oplus 2} 
\stackrel{(\phi_1, \phi_2)}{\to} \oO_D(-1) \to S_5(-1)[1]. 
\end{align*}
Here $\phi_1$, $\phi_2$ form a basis of 
$\Hom_{D}(\uU(-1), \oO_D(-1))$
as $\mathbb{C}[T]/T^2$-module. 

We show that $S_5$ must be a sheaf. 
Suppose that $S_5$ is not a sheaf, or equivalently
$\hH^0(S_5(-1)[1])$ is non-zero. 
By the construction of $\phi_i$, the composition 
\begin{align*}
\oO_D(-3C) \hookrightarrow \uU(-1) \stackrel{\phi_i}{\to} \oO_D(-1)
\end{align*}
is non-zero. 
Therefore we have 
\begin{align*}
\Cok(\phi_i) \cong \Cok(\phi_i' \colon \oO_D(-3C) \to \oO_C(-1))
\end{align*}
for some non-zero morphism $\phi_i'$. 
This implies that $\Cok(\phi_i)$ is zero dimensional, 
hence $\hH^0(S_5(-1)[1])$
is also a zero dimensional sheaf.
On the other hand, 
by the quiver description (\ref{quiver})
of $\modu(A_{D}^{\dag})$, it is easy 
to see that there 
is no morphism from 
$\Phi_{D}^{\dag}(S_5(-1)[1])$
to $\Phi_{D}^{\dag}(\oO_x)$
for any $x\in D$. 
This is a contradiction since
the natural morphism
\begin{align*}
S_5(-1)[1] \to \hH^0(S_5(-1)[1])
\end{align*} 
also becomes a zero map. 
Therefore $S_5$ is a sheaf 
(indeed a vector bundle) on $D$ which 
fits into an exact sequence (\ref{SUO}). 
\end{proof}
\begin{rmk}
The above analysis of the derived 
category of the singular quadric $D \subset \mathbb{P}^3$
shows that there is a semi orthogonal decomposition
\begin{align*}
D^b \Coh(D) =\langle \oO_D, \oO_D(-1), D^b \modu \left(
\mathbb{C}[T]/T^2 \right) \rangle. 
\end{align*}
This fact seems to be not seen in literatures. 
\end{rmk}

\subsection{Simple objects in $\PPer_{0}(X/Y)$}
In this subsection, we determine the 
simple objects in $\PPer_{0}(X/Y)$, 
where $\PPer_{0}(X/Y)$ is 
defined in (\ref{per:inc2}). 
We first recall the result by Van den Bergh~\cite{MVB}. 
\begin{prop}\label{fib:one}\emph{(\cite{MVB})}
Suppose that $f(D)$ is a curve. 
Then $\PPer_{0}(X/Y)$ is 
described by simple objects as 
\begin{align}\label{PL}
\PPer_{0}(X/Y)=\langle \oO_x, 
\oO_{L_y}(-2)[1], \oO_{L_y}(-1) : 
x\in X \setminus D, y\in f(D)\}.
\end{align}
Here $L_y\cneq f^{-1}(y) \cong \mathbb{P}^1$. 
\end{prop}

Below we assume that $f(D)$ is a point. 
We need the following
property of $\PPer(X/Y)$, which 
will be also used in the next section:
\begin{prop}\label{prop:property}
Suppose that $f(D)$ is a point. 
For $E\in \PPer(X/Y)$, we have 

(i) $\hH^i(E)=0$ for $i\notin \{-2, -1, 0\}$.

(ii) $\hH^{-1}(E)$, $\hH^{-2}(E)$ are supported on $D$. 

(iii) $\hH^{0}(E), \hH^{-2}(E)[2] \in \PPer(X/Y)$.
\end{prop}
\begin{proof}
The statement is local on $Y$, so we may 
assume that $Y$ is affine. Let us take a general 
member $H \in \lvert \oO_X(-D) \rvert$. 
We have the distinguished triangle
\begin{align*}
\dR f_{\ast}E \to \dR f_{\ast} \left(E(-D)\right) \to \dR f_{\ast} 
\left( E|_{H}(-D) \right).
\end{align*}
Since $\Phi_0(E) \in \Coh(\mathrsfs{A}_0)$, the left
two objects are coherent sheaves on $Y$. 
Therefore
\begin{align*}
R^i f_{\ast} \left( E|_{H}(-D) \right)=0, \quad i\neq -1, 0.
\end{align*} 
On the other hand, we have the spectral sequence
\begin{align*}
E_{2}^{p, q}=R^p f_{\ast} \left( \hH^q(E)|_{H}(-D) \right) \Rightarrow 
R^{p+q}f_{\ast} \left( E|_{H}(-D) \right)
\end{align*}
which degenerates since the fibers of $f|_{H}$ are 
at most one dimensional. (cf.~\cite[Lemma~3.1]{Br1}.)
Therefore we have
\begin{align}\label{one:dim}
R^p f_{\ast}\left(\hH^{q}(E)|_{H}(-D)\right)=0
\end{align}
if $p+q \neq -1, 0$, or equivalently
\begin{align*}
(p, q) \neq (0, 0), (0, -1), (1, -1), (1, -2).
\end{align*}
Hence for $q\notin \{-2, -1, 0\}$, we have the isomorphism
\begin{align}\label{isom:HE}
\dR f_{\ast} \hH^q(E) \stackrel{\sim}{\to}
\dR f_{\ast}\left(\hH^q(E)(-D)\right). 
\end{align}
We show that (\ref{isom:HE}) vanishes for 
$q\notin \{-2, -1, 0\}$. 
Let us consider the spectral sequence
\begin{align}\label{spe:f}
E_{2}^{p, q}=
R^p f_{\ast} \hH^q(E) \Rightarrow R^{p+q}f_{\ast}E. 
\end{align}
Since $\dR f_{\ast}E \in \Coh(Y)$, 
the morphism
\begin{align}\label{mor:spe}
f_{\ast} \hH^q(E) \to R^2 f_{\ast}\hH^{q-1}(E)
\end{align}
is an isomorphism   
for $q\notin \{-1, 0\}$, injective for $q=-1$ and surjective for $q=0$. 
Also we have 
\begin{align*}
R^1 f_{\ast}\hH^q(E)=0, \quad q\neq -1.
\end{align*}
Hence if we have
\begin{align}\label{if:have}
R^2 f_{\ast}\hH^q(E)=0, \quad q\le -3, q\ge 0, 
\end{align}
then the vanishing of (\ref{isom:HE}) for $q\notin \{-2, -1, 0\}$
follows.  
The condition (\ref{if:have}) follows from
the vanishing (\ref{one:dim})
for $p=1$, $q\le -3$, $q\ge 0$ and 
the duality argument given in the proof of~\cite[Theorem~6.1]{TU}.

By the vanishing of (\ref{isom:HE})
for $q\notin \{-2, -1, 0\}$, we have
\begin{align*}
\hH^{q}(E) \in \cC^{0}, \quad q\notin \{-2, -1, 0\}.
\end{align*}
Therefore we must have $\hH^q(E)=0$ for $q\notin \{-2, -1, 0\}$
by the definition of $\PPer(X/Y)$. So (i) is proved. 
The result of (ii) follows from that (\ref{mor:spe}) is injective 
for $q=-2, -1$. 

Now since (i) is proved, 
the spectral sequence (\ref{spe:f})
and a similar one for $\hH^0(E)(-D)$ shows 
\begin{align*}
R^i f_{\ast}\hH^0(E)=R^i f_{\ast}\left(\hH^0(E)(-D)\right)=0, 
\quad i\ge 1.
\end{align*}
Hence we have $\Phi_0(\hH^0(E)) \in \Coh(\mathrsfs{A}_0)$. 
Since we have 
\begin{align*}
\Hom(\cC^{<0}, \hH^0(E))=\Hom(\cC^{>0}, \hH^0(E))=0,
\end{align*}
it follows that $\hH^0(E) \in \PPer(X/Y)$. 
A similar argument also shows $\hH^{-2}(E)[2] \in \PPer(X/Y)$. 
\end{proof}

Using the notation of Corollary~\ref{cor:cor1}
and Corollary~\ref{cor:cor2}, we show the following:  
\begin{lem}\label{lem:rest}
Suppose that $f(D)$ is a point. Then 
the restriction morphism of algebras
\begin{align}\label{mor:rest}
\End(\widehat{\eE}, \widehat{\eE}) \to \End(\eE_{D}, \eE_{D})
\end{align}
is surjective. In particular, the inclusion 
$i \colon D \hookrightarrow X$ induces a fully faithful 
functor
\begin{align}\label{FF}
i_{\ast} \colon \PPer(D) \hookrightarrow \PPer_{0}(X/Y). 
\end{align}
\end{lem}
\begin{proof}
As for the first statement, 
it is enough to show 
$\Ext_{\widehat{X}}^1(\widehat{\eE}(-D), \widehat{\eE})=0$. 
By the formal function theorem, we have 
\begin{align*}
\Ext_{\widehat{X}}^1(\widehat{\eE}(-D), \widehat{\eE})
\cong \lim_{\longleftarrow } \Ext_{D_n}^1(\widehat{\eE}(-D)|_{D_n}, 
\widehat{\eE}|_{D_n}). 
\end{align*}
Here $D_n$ is the divisor $nD$. 
Applying $\Hom_{\widehat{X}}(\widehat{\eE}(-D), \ast)$
to the exact sequence
\begin{align*}
0 \to \eE_{D}(-nD) \to \widehat{\eE}|_{D_{n+1}} \to \widehat{\eE}|_{D_n} \to 0,
\end{align*}
it is enough to show 
\begin{align*}
\Ext_{D}^1(\eE_{D}, \eE_{D}(-nD))=0, \quad n\ge 0. 
\end{align*}
The above vanishing holds if $\eE_{D}(-nD) \in \PPer(D)$, 
since $\eE_{D}$ is a projective object of $\PPer(D)$. 
Using the exact sequence (\ref{give:ext}), it is easy to check that 
\begin{align*}
\dR \Hom_{D}(\oO_D \oplus \oO_D(D), \eE_{D}(-nD)) \in \modu (A_{D})
\end{align*}
for $n\ge 0$. 
Noting $\eE_{D}(-nD) \in \Coh(D)$, it follows that 
$\eE_{D}(-nD)$ is an object in $\PPer(D)$. 

As for the latter statement, the morphism (\ref{mor:rest}) 
induces a functor
\begin{align*}
\modu (A_{D}) \to \modu (\widehat{A})
\end{align*}
which is fully faithful since (\ref{mor:rest}) is surjective. 
Applying $\Phi_D$ and $\widehat{\Phi}$, 
we see that $i_{\ast} \colon D^b \Coh(D) \to D^b \Coh(\widehat{X})$
restricts to a fully faithful functor
\begin{align}\label{com1}
i_{\ast} \colon \PPer(D) \hookrightarrow \PPer_{0}(\widehat{X}/\widehat{Y}). 
\end{align}
On the other hand, 
the morphism of schemes $\widehat{X} \to X$ 
obviously induces a fully faithful functor 
\begin{align}\label{com2}
\PPer_{0}(\widehat{X}/\widehat{Y}) \hookrightarrow \PPer_{0}(X/Y). 
\end{align}
By (\ref{com1}) and (\ref{com2}), it follows that 
$i_{\ast}$ induces the 
fully faithful functor (\ref{FF}). 
\end{proof}
\begin{rmk}\label{rmk:subca}
By the surjection (\ref{mor:rest}), it follows that the 
subcategory 
\begin{align}\label{subca:P}
\PPer(D) \subset \PPer_{0}(\widehat{X}/\widehat{Y})
\end{align}
via (\ref{com1}) 
is closed under subobjects and quotients. 
\end{rmk}
We also prepare another lemma. 
\begin{lem}\label{ODmo}
Let $E \in \PPer_{0}(\widehat{X}/\widehat{Y})$ be a simple 
object. Then $\hH^i(E)$ is $\oO_D$-module 
for all $i\in \mathbb{Z}$. 
\end{lem}
\begin{proof}
For $x \in m_0 \subset \oO_{\widehat{Y}}$, we have 
the morphism 
\begin{align*}
E \stackrel{x}{\to} E. 
\end{align*}
Since $E$ is simple and supported on $D$, 
the above morphism must be a zero map. 
In particular, the induced morphism 
on $\hH^{i}(E)$ is also a zero map. 
Since $\Cok(f^{\ast}m_0 \to \oO_{\widehat{X}})=\oO_D$, 
this implies that each $\hH^i(E)$ is 
$\oO_D$-module. 
\end{proof}
Using the above lemmas, we show the following: 
\begin{prop}\label{prop:Persim}
Suppose that $f(D)$ is a point. 
Then 
$E \in \PPer_{0}(X/Y)$ is a simple object
if and only if $E \cong \oO_x$ for $x\notin D$ or 
$E \cong i_{\ast}F$ for a simple object
$F \in \PPer(D)$, where $i_{\ast}$ is 
given in (\ref{FF}). 
In particular, we have 
\begin{align*}
\PPer_{0}(X/Y) = \langle \oO_x, i_{\ast} \PPer(D) :
x \in X \setminus D \rangle.
\end{align*}
\end{prop}
\begin{proof}
It is obvious that 
an object $\oO_x$ for $x\notin D$ is a simple object in $\PPer_{0}(X/Y)$. 
Let $E \in \PPer_{0}(X/Y)$ be a simple 
object which is not isomorphic to $\oO_x$ for any $x \notin D$. 
Then $E$ is supported on $D$, hence 
$E \in \PPer_{0}(\widehat{X}/\widehat{Y})$. 
By Lemma~\ref{lem:rest} and Remark~\ref{rmk:subca}, it is enough
to show that $E$ is contained in the subcategory (\ref{subca:P}).  

First suppose that $\hH^0(E)$ is non-zero. 
By Proposition~\ref{prop:property},
$\hH^0(E) \in \PPer_{0}(\widehat{X}/\widehat{Y})$
and 
we have the non-zero morphism 
\begin{align*}
E \to \hH^0(E).
\end{align*}
The above morphism
 must be injective in $\PPer_{0}(\widehat{X}/\widehat{Y})$, 
since $E$ is a simple object in $\PPer_{0}(\widehat{X}/\widehat{Y})$.  
By Lemma~\ref{ODmo}, the sheaf $\hH^0(E)$ is $\oO_D$-module, 
hence $\hH^0(E) \in \PPer(D)$. 
Then we have $E \in \PPer(D)$ by Remark~\ref{rmk:subca}.

Next suppose that $\hH^0(E)=0$ and $\hH^{-2}(E) \neq 0$. 
By the same argument as above, we have the non-zero morphism
\begin{align*}
\hH^{-2}(E)[2] \to E
\end{align*} 
which must be surjective in $\PPer_{0}(\widehat{X}/\widehat{Y})$. 
Hence we have
$E \in \PPer(D)$ by Lemma~\ref{ODmo} and Remark~\ref{rmk:subca}. 
Also if $\hH^0(E)=\hH^{-2}(E)=0$, then 
$E=\hH^{-1}(E)[1]$ and $E \in \PPer(D)$ follows from 
Lemma~\ref{ODmo}. 

Finally $\PPer_{0}(X/Y)$ is a finite length abelian category 
by Theorem~\ref{thm:tilting}, so it is generated by 
its simple objects. The last statement follows from this fact. 
\end{proof}

\section{Conjectural Bridgeland stability conditions}
In this section, 
we prove Theorem~\ref{thm:intro2} 
using the results in 
the previous section. 
Throughout this section, 
$X$ is a smooth projective 3-fold, 
$f \colon X \to Y$ is an extremal contraction
and $D \subset X$ is the exceptional divisor.  

\subsection{First tilting of $\PPer(X/Y)$}\label{sec:conj}
For an ample divisor $\omega$ on $Y$, 
we consider $f^{\ast}\omega$-slope function on 
$\PPer(X/Y)$ and apply a similar construction 
as in Subsection~\ref{subsec:const}. 

For $E\in \PPer(X/Y)$, the $f^{\ast}\omega$-slope is defined by 
\begin{align*}
\mu_{f^{\ast}\omega}(E)
=\frac{\ch_1^{}(E) \cdot (f^{\ast}\omega)^2}{\ch_0(E)}.
\end{align*}
Here $\mu_{f^{\ast}\omega}(E)=\infty$ when $\ch_0(E)=0$. 
Similarly to Definition~\ref{defi:slope}, 
we can define the $\mu_{f^{\ast}\omega}$-stability
on $\PPer(X/Y)$, 
which has the Harder-Narasimhan property: 
\begin{lem}
Any object in $\PPer(X/Y)$
admits a Harder-Narasimhan filtration with respect to 
$\mu_{f^{\ast}\omega}$-stability. 
\end{lem}
\begin{proof}
By Theorem~\ref{thm:tilting}, 
the abelian category $\PPer(X/Y)$ is noetherian, 
and the lemma is proved along with the same argument of Lemma~\ref{lem:HN}. 
\end{proof}
By the above lemma, 
the torsion pair $(\tT_{f^{\ast}\omega}, \fF_{f^{\ast}\omega})$
is defined exactly as in the same way of Subsection~\ref{subsec:const}. 
Its tilting is similarly defined,
\begin{align*}
\bB_{f^{\ast}\omega} =\langle \fF_{f^{\ast}\omega}[1], 
\tT_{f^{\ast}\omega} \rangle. 
\end{align*}
We have the following lemma:
\begin{lem}\label{noet:B}
The abelian category $\bB_{f^{\ast}\omega}$ is noetherian. 
\end{lem}
\begin{proof}
The statement 
follows from the same argument of~\cite[Proposition~7.1]{Brs2}, 
with a minor modification. 
Suppose that there is an infinite sequence of surjections 
in $\bB_{f^{\ast}\omega}$,
\begin{align}\label{seq:ter}
E=E_1 \twoheadrightarrow E_2 \twoheadrightarrow \cdots
\twoheadrightarrow E_{i} \twoheadrightarrow E_{i+1} \twoheadrightarrow \cdots. 
\end{align}
Let us take the exact sequence in $\bB_{f^{\ast}\omega}$
\begin{align*}
0 \to L_i \to E \to E_{i} \to 0. 
\end{align*}
Since $\ch_1(\ast) \cdot (f^{\ast}\omega)^2 \ge 0$
on $\bB_{f^{\ast}\omega}$, we may assume that 
$\ch_1(E_i) \cdot (f^{\ast}\omega)^2$ is constant. 
Hence $\ch_1(L_i) \cdot (f^{\ast}\omega)^2=0$, 
which implies that 
\begin{align}\label{mu=0}
L_i \in \left\langle F[1], T : 
\begin{array}{c}  
F\in \PPer(X/Y) \mbox{ is } \mu_{f^{\ast}\omega} \mbox{-semistable with } \\
\mu_{f^{\ast}\omega}(F)=0, \
T \in \PPer_{\le 1}(X/Y).
\end{array} \right\rangle. 
\end{align}
 Applying the same argument 
of~\cite[Proposition~7.1]{Brs2},
replacing $\Coh(X)$ by $\PPer(X/Y)$, 
we arrive at the exact sequences in $\PPer(X/Y)$
\begin{align*}
0 \to Q \to \hH_p^{-1}(E_i) \to \hH_p^{0}(L_i) \to 0,
\end{align*}
where $Q \in \fF_{f^{\ast}\omega}$ is independent of $i$
and inclusions in $\PPer(X/Y)$
\begin{align*}
\hH_p^{-1}(E_1) \subset \hH_p^{-1}(E_{2}) \subset \cdots 
\subset \hH_p^{-1}(E_i) \subset \hH_p^{-1}(E_{i+1}) \subset \cdots.
\end{align*}
We apply the equivalence $\Phi$ in Theorem~\ref{thm:tilting}, 
and forget the $\mathrsfs{A}$-module structures. 
We obtain the exact sequence in $\Coh(Y)$
\begin{align*}
0 \to \Phi(Q) \to \Phi(\hH_p^{-1}(E_i)) \to \Phi(\hH_p^0(L_i)) \to 0. 
\end{align*}
Since $\hH_p^0(L_i) \in \PPer_{\le 1}(X/Y)$, 
the sheaf $\Phi(\hH_p^0(L_i))$ is at most one dimensional. 
On the other hand, since $Q \in \fF_{f^{\ast}\omega}$
satisfies 
\begin{align*}
\Hom(\PPer_{\le 2}(X/Y), Q)=0,
\end{align*} 
the sheaf $\Phi(Q)$ is a torsion free sheaf on $Y$. 
Similarly $\Phi(\hH_p^{-1}(E_i))$ is also torsion free, 
so we have the sequence in $\Coh(Y)$
\begin{align*}
\Phi(Q) \subset \Phi(\hH_p^{-1}(E_1)) \subset \Phi(\hH_p^{-1}(E_2))
 \subset \cdots \subset \Phi(Q)^{\vee \vee}. 
\end{align*}
The above sequence terminates since $\Coh(Y)$ is noetherian. 
Therefore the sequence (\ref{seq:ter}) also terminates. 
\end{proof}

\subsection{Second tilting of $\PPer(X/Y)$} 
In this subsection, we
tilt $\bB_{f^{\ast}\omega}$ again. 
This is an analogy of the
double tilting of $\Coh(X)$ in~\cite{BMT}
for $\PPer(X/Y)$.

We first show the following lemma, 
whose proof relies on the results in the 
previous section. 
\begin{lem}\label{lem:key} There is $b\in \mathbb{Q}$ such that
we have 
\begin{align}\label{ineq:chB}
\ch_3^{bD}(E)>0
\end{align}
for any $0\neq E \in \PPer_{0}(X/Y)$. 
\end{lem}
\begin{proof}
Suppose that $f(D)$ is a curve. 
Then by calculating Chern characters 
of simple objects
in the RHS of (\ref{PL}), 
the desired inequality (\ref{ineq:chB}) is checked
to hold 
when 
\begin{align*}
\frac{1}{2}< b < \frac{3}{2}. 
\end{align*}

Next suppose that $f(D)$ is a point. 
 By Proposition~\ref{prop:Persim}, 
it is enough to find $b\in \mathbb{Q}$ so 
that the inequality (\ref{ineq:chB}) holds for 
any object $E=i_{\ast}F$, where
$F \in \PPer(D)$ is a simple object. 
Similarly to the case that $f(D)$ is a curve, 
such $b$ can be found by computing 
Chern characters of simple objects in $\PPer(D)$
appearing in Proposition~\ref{prop:PDsim}. 

We give some more detail on the computation for 
the type V case. We have
\begin{align*}
&\ch(i_{\ast}\oO_D(-1))= \left( 0, D, -C, 1/3 \right), \\
&\ch(i_{\ast}S_5)=(0, 3D, -C, -1), \\
&\ch(i_{\ast}\oO_D(-C))=(0, D, 0, -1/6). 
\end{align*}
Therefore, setting $c=1-b$, we have
\begin{align*}
&\ch^{bD}_{3}(i_{\ast}\oO_D(-2)[2])=
c^2 + c+ 1/3, \\
&\ch^{bD}_{3}(i_{\ast}S_5(-1)[1])=-3c^2 - c+1, \\
& \ch^{bD}_{3}(i_{\ast}\oO_D(-3C))=c^2 -1/6. 
\end{align*}
We want to find $b$ so that all the 
above values are positive. The 
solution of the inequalities is 
\begin{align*}
1+\frac{\sqrt{6}}{6}< b < \frac{7}{6} +\frac{\sqrt{13}}{6} \quad \mbox{ or }
\quad
\frac{7}{6}-
\frac{\sqrt{13}}{6}<b< 1- \frac{\sqrt{6}}{6}.
\end{align*}
The other cases are similarly calculated. 
The result is as follows: 
in type III case, the result is the same as in 
type V case, and 
\begin{align*}
&2-\frac{\sqrt{6}}{3}< b< 2+\frac{\sqrt{6}}{3} \quad \mbox{ in type II, } \\
&\frac{3}{4}+ \frac{\sqrt{15}}{12}< b< 
\frac{4}{5}+ \frac{\sqrt{6}}{6} \quad \mbox{ in type IV. }
\end{align*}
\end{proof}

Let us take $b\in \mathbb{Q}$
as in Lemma~\ref{lem:key},
and set $B=bD$. 
Note that we have 
\begin{align}\label{note:B}
\ch_i^{B}(\ast) (f^{\ast}\omega)^{3-i} = \ch_i(\ast)(f^{\ast}\omega)^{3-i}, 
\ 0 \le i \le 1.
\end{align}
Next we generalize Bogomolov-Gieseker inequality
for perverse coherent sheaves 
to our 3-fold situation. 
\begin{prop}\label{prop:BG3fold}
For any $\mu_{f^{\ast}\omega}$-semistable 
object $E \in \PPer(X/Y)$ with $\ch_0(E)>0$,
we have the inequality
\begin{align}\label{BG:3}
\left\{\ch_1(E) \cdot  (f^{\ast}\omega)^2 \right\}^2
\ge 2\omega^3 \ch_0(E) (\ch_2^{B}(E) \cdot f^{\ast}\omega).
\end{align}
\end{prop}
\begin{proof}
First suppose that $f(D)$ is a point.
Note that the equality (\ref{note:B}) also holds for $i=2$
in this case. 
By Proposition~\ref{prop:property}, 
$\hH^0(E)$ is torsion free outside $D$.
Since sheaves supported on $D$ do not affect
$\mu_{f^{\ast}\omega}$, the free part of $\hH^0(E)$
is a $\mu_{f^{\ast}\omega}$-semistable sheaf. 
Then the desired inequality (\ref{BG:3}) follows from 
the Bogomolov-Gieseker inequality for the 
free part of $\hH^0(E)$. 

Next suppose that $f(D)$ is a curve.
Similarly to the proof of Proposition~\ref{prop:BGper}, 
$E$ is a coherent sheaf. Also there is an 
exact sequence
\begin{align*}
0 \to T \to E \to F \to 0
\end{align*}
such that
$F$ is a torsion free $\mu_{f^{\ast}\omega}$-semistable 
sheaf, $T$ is a torsion sheaf supported on $D$ 
with $T[1] \in \PPer(X/Y)$.
By Theorem~\ref{thm:BoGi}, the Hodge index theorem
and noting (\ref{note:B}), we have 
\begin{align*}
\left\{ \ch_1(F) \cdot (f^{\ast}\omega)^2 \right\}^2 /\omega^3
&\ge \ch_1^{B}(F)^2 f^{\ast}\omega \\
&\ge 2\ch_0(F) \ch_2^{B}(F)f^{\ast}\omega. 
\end{align*} 
 By the argument as in Remark~\ref{rmk:weak}, 
it is enough to show that 
$\ch_2^{B}(T) \cdot f^{\ast}\omega \le 0$. 

By Theorem~\ref{thm:tilting}, the condition 
$T[1] \in \PPer(X/Y)$ is equivalent to 
\begin{align*}
f_{\ast}T =f_{\ast} \left(T(-D) \right) =0. 
\end{align*}
Let $W \subset Y$ be a divisor which is linearly equivalent 
to some multiple of $\omega$, and $Z \cneq f^{-1}(W)$. 
If we take $W$ to be smooth and general, then 
$f|_{Z} \colon Z \to W$ is a blow up at a finite number of points, 
and 
\begin{align*}
f|_{Z\ast}(T|_{Z}) =f|_{Z\ast} \left(T|_{Z}(-D) \right) =0.
\end{align*}
 Then 
$T|_{Z}[1] \in \PPer_{0}(X/Y)$
by Theorem~\ref{thm:tilting}, 
so Lemma~\ref{lem:key} implies
\begin{align*}
\ch_2^{B}(T[1]) \cdot Z = 
\ch_3^{B}(T|_{Z}[1])>0. 
\end{align*} 
Therefore $\ch_2^{B}(T) \cdot f^{\ast}\omega \le 0$ follows. 
\end{proof}
We consider the central charge $Z_{B, f^{\ast}\omega}$
given by (\ref{integral}). 
Noting (\ref{note:B}), 
$Z_{B, f^{\ast}\omega}(E)$ is 
written as  
\begin{align*}
\left(-\ch_3^{B}(E)
 + \frac{(f^{\ast}\omega)^2}{2} \ch_1^{}(E) \right) + 
\sqrt{-1} \left(f^{\ast}\omega \ch_2^{B}(E)
- \frac{\omega^3}{6}\ch_0^{}(E)  \right).
\end{align*}
As an analogy of~\cite[Lemma~3.2.1]{BMT}, we have the following lemma:
\begin{lem}\label{triple}
For any non-zero object
$E\in \bB_{f^{\ast}\omega}$, one of 
the following conditions hold:

(a) $\ch_1(E) \cdot (f^{\ast}\omega)^2 >0$.

(b) $\ch_1(E) \cdot (f^{\ast}\omega)^2 =0$
and $\Imm Z_{B, f^{\ast}\omega}(E)>0$. 

(c) $\ch_1(E) \cdot (f^{\ast}\omega)^2 = 
\Imm Z_{B, f^{\ast}\omega}(E)=0$ and 
$\Ree Z_{B, f^{\ast}\omega}(E) <0$. 
\end{lem}
\begin{proof}
By the construction of $\bB_{f^{\ast}\omega}$, 
we always have $\ch_1(E) \cdot (f^{\ast}\omega)^2 \ge 0$. 
Suppose that $\ch_1(E) \cdot (f^{\ast}\omega)^2 =0$. 
Then $E$ is contained in the RHS of (\ref{mu=0}). 
If $F$ is a $\mu_{f^{\ast}\omega}$-semistable 
object in $\PPer(X/Y)$ with $\mu_{f^{\ast}\omega}(F)=0$, 
then $\ch_2^{B}(F) \cdot f^{\ast}\omega \le 0$
by Proposition~\ref{prop:BG3fold}. 
Hence $\Imm Z_{B, f^{\ast}\omega}(F[1]) >0$
holds. Also for an object $T \in \PPer_{\le 1}(X/Y)$, 
we have $\ch_2(T) \cdot f^{\ast}\omega \ge 0$, 
hence $\Imm Z_{B, f^{\ast}\omega}(T) \ge 0$. 
Then the inequality $\Imm Z_{B, f^{\ast}\omega}(E) \ge 0$
follows 
from these inequalities. 
Finally if $\ch_1(E) \cdot (f^{\ast}\omega)^2
= \Imm Z_{B, f^{\ast}\omega}(E)=0$, the above 
argument shows $E \in \PPer_{0}(X/Y)$. 
Hence $\mathrm{Re}Z_{B, f^{\ast}\omega}(E)<0$
follows from our choice of $B=bD$
 that the condition of Lemma~\ref{lem:key} 
is satisfied. 
\end{proof}
We define 
the slope function $\nu_{B, f^{\ast}\omega}$ on $\bB_{f^{\ast}\omega}$
 to be
\begin{align*}
\nu_{B, f^{\ast}\omega}(E)=
\frac{\Imm Z_{B, f^{\ast}\omega}(E)}{\ch_1(E) \cdot (f^{\ast}\omega)^2}.
\end{align*} 
Here $\nu_{B, f^{\ast}\omega}(E)=\infty$ 
if $\ch_1(E) \cdot (f^{\ast}\omega)^2=\infty$. 
By Lemma~\ref{triple},
the above slope function is an analogy of 
the slope function for torsion free sheaves on algebraic surfaces. 
Also a similar slope function was defined
on a tilting of $\Coh(X)$ in~\cite{BMT}. 
The following is the analogy of `tilt stability'
in~\cite{BMT} for the tilting of perverse t-structure. 
\begin{defi}
A non-zero object $E\in \bB_{f^{\ast}\omega}$ is 
$\nu_{B, f^{\ast}\omega}$-(semi)stable if 
for any exact sequence $0 \to F \to E \to G \to 0$
in $\bB_{f^{\ast}\omega}$, we have the inequality
\begin{align*}
\nu_{B, f^{\ast}\omega}(F) <(\le) \nu_{B, f^{\ast}\omega}(G). 
\end{align*}
\end{defi}
\begin{lem}\label{B:HN}
Any object in $\bB_{f^{\ast}\omega}$ admits 
a Harder-Narasimhan filtration with respect to 
$\nu_{B, f^{\ast}\omega}$-stability. 
\end{lem}
\begin{proof}
Since $\bB_{f^{\ast}\omega}$ is noetherian 
by Lemma~\ref{noet:B}, 
the same argument of~\cite[Proposition~7.1]{Brs2}
is applied.  
\end{proof}
We consider the following subcategories of $\bB_{f^{\ast}\omega}$:
\begin{align*}
\tT_{B, f^{\ast}\omega}' & \cneq
 \langle E : E \mbox{ is }\nu_{B, f^{\ast}\omega}
\mbox{-semistable with } \nu_{B, f^{\ast}\omega}(E)>0 \rangle, \\
\fF_{B, f^{\ast}\omega}' & \cneq
 \langle E : E \mbox{ is }\nu_{B, f^{\ast}\omega}
\mbox{-semistable with } \nu_{B, f^{\ast}\omega}(E) \le 0 \rangle. 
\end{align*}
By Lemma~\ref{B:HN}, 
the pair of subcategories 
$(\tT_{B, f^{\ast}\omega}', \fF_{B, f^{\ast}\omega}')$
forms a torsion pair on $\bB_{f^{\ast}\omega}$. 
By tilting, we have the following heart of a t-structure:
\begin{defi}
We define $\aA_{B, f^{\ast}\omega} 
\subset D^b \Coh(X)$ to be
\begin{align*}
\aA_{B, f^{\ast}\omega} \cneq \langle \fF_{B, f^{\ast}\omega}'[1], 
\tT_{B, f^{\ast}\omega}' \rangle. 
\end{align*}
\end{defi}
Note that, by 
the construction of $\aA_{B, f^{\ast}\omega}$, we have 
\begin{align}\label{IZg}
Z_{B, f^{\ast}\omega}(\aA_{B, f^{\ast}\omega} \setminus \{0\} )
\subset \{ z \in \mathbb{C} : \Imm z \ge 0\}. 
\end{align}
\subsection{Conjectures}
Let us consider the pair
\begin{align*}
\sigma_{B, f^{\ast}\omega}
=(Z_{B, f^{\ast}\omega}, \aA_{B, f^{\ast}\omega}). 
\end{align*}
Similarly to~\cite[Conjecture~3.2.6]{BMT}, we
propose the following conjecture: 
\begin{conj}\label{conj:stab}
We have 
$\sigma_{B, f^{\ast}\omega} \in \Stab(X)$. 
\end{conj}
\begin{rmk}
For $0<\varepsilon \ll 1$
so that $f^{\ast}\omega -\varepsilon D$ is ample, 
a conjectural Bridgeland stability condition 
\begin{align}\label{BMT}
\sigma_{B, f^{\ast}\omega -\varepsilon D}
=(Z_{B, f^{\ast}\omega -\varepsilon D}, 
\aA_{B, f^{\ast}\omega -\varepsilon D})
\end{align}
is constructed in~\cite{BMT}. 
If both of~\cite[Conjecture~3.2.6]{BMT}
and Conjecture~\ref{conj:stab} are true, then
we conjecture that 
\begin{align}\label{BMT:lim}
\lim_{\varepsilon \to +0}
\sigma_{B, f^{\ast}\omega -\varepsilon D}
=\sigma_{B, f^{\ast}\omega}
\end{align} 
in $\Stab(X)$. 
The above equality should follow from the 
support property of $\sigma_{B, f^{\ast}\omega}$. 
Proving this would require 
further evaluations of 
Chern classes of $\sigma_{B, f^{\ast}\omega}$-semistable 
objects, as we discussed in Section~\ref{sec:surface}
for surfaces. 
\end{rmk}
If Conjecture~\ref{conj:stab} is true, 
we are interested in the subcategory of 
semistable objects of phase one. 
Although we are not able to prove Conjecture~\ref{conj:stab}
yet, such a subcategory is well-defined
as follows: 
\begin{defi}\label{def:P(1)}
We define $\pP_{B, f^{\ast}\omega}(1)$ to be
\begin{align*}
\pP_{B, f^{\ast}\omega}(1) \cneq 
\{ E \in \aA_{B, f^{\ast}\omega} :
\Imm Z_{B, f^{\ast}\omega}(E)=0\}. 
\end{align*}
\end{defi}
Note that 
$\pP_{B, f^{\ast}\omega}(1)$ is an abelian subcategory of 
$\aA_{B, f^{\ast}\omega}$.  
The construction of $\aA_{B, f^{\ast}\omega}$
immediately implies 
\begin{align}\label{IZ=0}
\pP_{B, f^{\ast}\omega}(1) = 
\left\langle F[1], T : 
\begin{array}{c}  
F\in \bB_{f^{\ast}\omega}
 \mbox{ is } \nu_{B, f^{\ast}\omega} \mbox{-semistable with } \\
\nu_{B, f^{\ast}\omega}(F)=0, \
T \in \PPer_{0}(X/Y).
\end{array} \right\rangle. 
\end{align}

\begin{rmk}
The abelian categories $\aA_{B, f^{\ast}\omega}$
and $\pP_{B, f^{\ast}\omega}(1)$ are independent of $B$
if $f(D)$ is a point.  
\end{rmk}

In order to show Conjecture~\ref{conj:stab}, we need to 
show the axiom (\ref{pro:1}). As in~\cite{BMT}, 
this condition is equivalent to the following 
Bogomolov-Gieseker type inequality evaluating $\ch_3$:
\begin{conj}\label{conj:BG}
Take $b\in \mathbb{Q}$ as in Lemma~\ref{lem:key}
and set $B=bD$. Then for any 
$\nu_{B, f^{\ast}\omega}$-semistable 
object $E \in \bB_{f^{\ast}\omega}$ with $\nu_{B, f^{\ast}\omega}(E)=0$, 
we have the inequality
\begin{align*}
\ch_3^{B}(E) < \frac{(f^{\ast}\omega)^2}{2}\ch_1(E). 
\end{align*}
\end{conj}
Indeed we have the following: 
\begin{prop}
Suppose that Conjecture~\ref{conj:BG} is true. 
Then we have $(Z_{B, f^{\ast}\omega}, \aA_{B, f^{\ast}\omega}) 
\in \Stab^{\dag}(X)$.  
\end{prop}
\begin{proof}
First we check (\ref{pro:1}).
By the condition (\ref{IZg}), it is enough to check 
that 
any non-zero object $E \in \pP_{B, f^{\ast}\omega}(1)$
satisfies $\Ree Z_{B, f^{\ast}\omega}(E)<0$.  
Since $E$ is contained in the RHS of (\ref{IZ=0}), 
the condition $\Ree Z_{B, f^{\ast}\omega}(E)<0$ follows from 
Conjecture~\ref{conj:BG} and our choice of $B=bD$ so that 
the condition of Lemma~\ref{lem:key} is satisfied. 

In Proposition~\ref{prop:noether} below, 
we show that the abelian category $\aA_{B, f^{\ast}\omega}$
is noetherian. The Harder-Narasimhan property 
follows from this fact and the same argument of Lemma~\ref{lem:HN}. 
The locally finiteness is obvious since 
the image of $Z_{B, f^{\ast}\omega}$ is a discrete subgroup
in $\mathbb{C}$. 
\end{proof}

\subsection{Finite length property}
This subsection
 is devoted to showing
some technical results:
noetherian property of 
$\aA_{B, f^{\ast}\omega}$ 
and finite length property of 
$\pP_{B, f^{\ast}\omega}(1)$.
Both of them are necessary conditions
for the Conjecture~\ref{conj:stab} to hold. 
\begin{prop}\label{prop:noether}
The abelian category $\aA_{B, f^{\ast}\omega}$ is noetherian. 
\end{prop}
\begin{proof}
We give a direct proof for this fact without 
passing to polynomial stability 
conditions as in~\cite[Proposition~5.2.2]{BMT}. 
Suppose that there is an infinite sequence of 
surjections in $\aA_{B, f^{\ast}\omega}$
\begin{align}\notag
E=E_1 \twoheadrightarrow E_2 \twoheadrightarrow \cdots
\twoheadrightarrow E_{i} \twoheadrightarrow E_{i+1} \twoheadrightarrow \cdots. 
\end{align}
Let us take the exact sequence in $\aA_{B, f^{\ast}\omega}$
\begin{align*}
0 \to L_i \to E \to E_i \to 0. 
\end{align*}
Since $\Imm Z_{B, f^{\ast}\omega}(\ast) \ge 0$
on $\aA_{B, f^{\ast}\omega}$, we may assume that 
$\Imm Z_{B, f^{\ast}\omega}(E_i)$ is 
constant. Hence $\Imm Z_{B, f^{\ast}\omega}(L_i)=0$, 
which implies that $L_i$ is contained in 
the RHS of (\ref{IZ=0}). 
Similarly to the proof of Lemma~\ref{B:HN}, 
if we apply the 
 same argument of~\cite[Proposition~7.1]{Brs2}, 
we arrive at the exact sequences in $\bB_{f^{\ast}\omega}$
\begin{align*}
0 \to Q \to \hH_{\bB}^{-1}(E_i) \to \hH_{\bB}^{0}(L_i) \to 0
\end{align*}
where $Q \in \fF_{B, f^{\ast}\omega}'$ is independent of $i$, 
and $\hH_{\bB}^{i}(\ast)$ is the $i$-th cohomology 
functor with respect to the t-structure with heart $\bB_{f^{\ast}\omega}$. 
We also have the inclusions in $\bB_{f^{\ast}\omega}$
\begin{align}\label{seq:ter2}
\hH_{\bB}^{0}(L_1) \subset \hH_{\bB}^{0}(L_2) \subset \cdots 
\subset \hH_{\bB}^{0}(L_i) \subset \hH_{\bB}^{0}(L_{i+1}) \subset \cdots.
\end{align}
Since $\hH_{\bB}^{0}(L_i) \in \PPer_{0}(X/Y)$, 
the object $\Phi(\hH_{\bB}^{0}(L_i))$ is a zero dimensional 
sheaf as $\oO_Y$-module. 
Hence it is enough to bound the length of $\Phi(\hH_{\bB}^{0}(L_i))$. 

In order to reduce the notation, we set 
\begin{align*}
V_i= \hH_{\bB}^{-1}(E_i), \quad T_i = \hH_{\bB}^{0}(L_i).
\end{align*}
By Lemma~\ref{lem:torsion} below, 
we have the exact sequences in $\bB_{f^{\ast}\omega}$
\begin{align*}
&0 \to Q^{(1)} \to Q \to Q^{(2)} \to 0 \\
&0 \to V_i^{(1)} \to V_i \to V_i^{(2)} \to 0
\end{align*}
which respect the torsion pair $(\tT^{\dag}, \fF^{\dag})$. 
Using the snake lemma, it is easy to see that 
there are exact sequences in $\bB_{f^{\ast}\omega}$
\begin{align}\label{seq:Ti}
&0 \to Q^{(j)} \to V_{i}^{(j)} \to T_i^{(j)} \to 0, \\
\notag
&0 \to T_{i}^{(1)} \to T_{i} \to T_{i}^{(2)} \to 0,
\end{align}
for some $T_i^{(j)} \in \PPer_{0}(X/Y)$ with
$j=1, 2$. 
Since $\Phi(T_i^{(j)})$ is a zero dimensional sheaf, 
it is enough to bound its length. 

First we bound the length of $\Phi(T_{i}^{(1)})$. 
By applying $\mathbb{D}\Phi$ given in Lemma~\ref{DPhi} below
to the sequence (\ref{seq:Ti}) for $j=1$, 
we obtain the distinguished triangle in $D^b \Coh(Y)$
\begin{align*}
\mathbb{D} \Phi(T_i^{(1)}) \to \mathbb{D} \Phi(V_{i}^{(1)}) \to 
\mathbb{D}\Phi(Q^{(1)}).
\end{align*}
Note that $Q^{(1)}, 
V_i^{(1)} \in \tT^{\dag} \cap \fF_{B, f^{\ast}\omega}'$. 
Therefore  
Lemma~\ref{DPhi}
implies that 
$\hH^2 (\mathbb{D}\Phi(Q^{(1)}))$
is zero dimensional, and 
 we obtain the exact
sequence in $\Coh(Y)$
\begin{align*}
\hH^2(\mathbb{D} \Phi(V_{i}^{(1)})) \to 
\hH^2(\mathbb{D}\Phi(Q^{(1)}))
 \to \hH^3 (\mathbb{D} \Phi(T_i^{(1)})) \to 0. 
\end{align*}
Since $\hH^3 (\mathbb{D} \Phi(T_i^{(1)}))$
is a zero dimensional sheaf whose length is equal 
to the length of $\Phi(T_i^{(1)})$, we obtain the 
bound of the length of $\Phi(T_i^{(1)})$. 

As for the bound of the length of $\Phi(T_i^{(2)})$, 
let $Q^{(3)}$, $V_i^{(3)}$ be the maximal subobjects of $Q^{(2)}$, 
$V_i^{(2)}$ in $\PPer(X/Y)$
contained in $\PPer_{\le 2}(X/Y)$, 
and set 
\begin{align*}
Q^{(4)}=Q^{(2)}/Q^{(3)}, \quad V_i^{(4)}=V_i^{(2)}/V_{i}^{(3)}.
\end{align*}
Similarly as above, we have the exact sequences in $\PPer(X/Y)$
\begin{align*}
&0 \to Q^{(j)} \to V_i^{(j)} \to T_i^{(j)} \to 0 \\
& 0 \to T_i^{(3)} \to T_i^{(2)} \to T_i^{(4)} \to 0,
\end{align*}
for some $T_{i}^{(j)} \in \PPer_{0}(X/Y)$
with $j=3, 4$. 
Since $\Phi(Q^{(4)})$ and $\Phi(V_i^{(4)})$
are torsion free sheaves on $Y$, 
the bound of the length of $\Phi(T_i^{(4)})$ is obtained since 
\begin{align*}
\Phi(T_i^{(4)}) \subset \Phi(Q^{(4)})^{\vee \vee}/\Phi(Q^{(4)}). 
\end{align*}
Also since $\Phi(Q^{(3)})$ and 
 $\Phi(V_i^{(3)})$ are pure two dimensional sheaves on $Y$, 
the bound of the length of $\Phi(T_i^{(3)})$ is obtained by 
taking the projection $Y \stackrel{\phi}{\dashrightarrow} \mathbb{P}^2$, 
which is defined and
 finite over the support of $\Phi(Q^{(3)})$, and noting 
that 
\begin{align*}
\phi_{\ast}\Phi(T_i^{(3)}) \subset \phi_{\ast} \Phi(Q^{(3)})^{\vee \vee}/ \phi_{\ast}\Phi(Q^{(3)}). 
\end{align*}

\end{proof}
We have used the following lemmas:
\begin{lem}\label{lem:torsion}
There is a torsion pair $(\tT^{\dag}, \fF^{\dag})$
on $\bB_{f^{\ast}\omega}$ such that 

\begin{itemize}
\item $E \in \tT^{\dag}$ if and only if 
$\hH_{p}^{-1}(E) \in \fF_{f^{\ast}\omega}$
and $\hH_p^{0}(E) \in \PPer_{\le 1}(X/Y)$. 

\item $E\in \fF^{\dag}$ if and only 
if $\hH_{p}^{-1}(E)=0$, $\hH_p^{0}(E) \in \tT_{f^{\ast}\omega}$
and 
\begin{align*}
\Hom(\PPer_{\le 1}(X/Y), \hH_p^{0}(E))=0.
\end{align*}
\end{itemize}
\end{lem}
\begin{proof}
If we define $(\tT^{\dag}, \fF^{\dag})$ in a required 
way, then it is obvious that $\Hom(\tT^{\dag}, \fF^{\dag})=0$. 
For any $E \in \bB_{f^{\ast}\omega}$, 
the decomposition of $E$ into objects in $\tT^{\dag}$ and $\fF^{\dag}$
is obtained by composing the exact sequence
in $\bB_{f^{\ast}\omega}$
\begin{align*}
0 \to \hH_p^{-1}(E)[1] \to E \to \hH_p^{0}(E) \to 0
\end{align*}
with the exact sequence
\begin{align*}
0 \to T_1 \to \hH_p^{0}(E) \to T_2 \to 0
\end{align*}
where $T_1$ is the maximal subobject of $\hH_p^{0}(E)$
in $\PPer(X/Y)$ contained in $\PPer_{\le 1}(X/Y)$. 
\end{proof}
\begin{lem}\label{DPhi}
For an object $E \in \tT^{\dag} \cap \fF_{B, f^{\ast}\omega}'$, 
the object
\begin{align*}
\mathbb{D}\Phi(E) \cneq \dR \hH om_{\oO_Y}(\Phi(E), \oO_Y) \in D^b \Coh(Y)
\end{align*}
satisfies that $\hH^1(\mathbb{D}\Phi(E))$ is a torsion 
free sheaf, $\hH^2(\mathbb{D}\Phi(E))$ is zero dimensional 
and $\hH^i(\mathbb{D}\Phi(E))=0$ for $i\neq 1, 2$. 
\end{lem}
\begin{proof}
We first show that 
\begin{align}\label{Coh1}
\Hom_{Y}(\Coh_{\le 1}(Y), \Phi(E))=0.
\end{align}
Suppose that there is $F \in \Coh_{\le 1}(Y)$
and a non-trivial morphism $F \to \Phi(E)$. 
By taking the adjunction, we have 
\begin{align}\label{adju:A}
\Hom_{\mathrsfs{A}}(F \dotimes_{\oO_Y} \mathrsfs{A}, \Phi(E)) \neq 0. 
\end{align}
Let us consider the object $\Phi^{-1}(F\dotimes_{\oO_Y}\mathrsfs{A})$. 
We have 
$\hH_{p}^{i}\Phi^{-1}(F\dotimes_{\oO_Y}\mathrsfs{A})=0$
for $i \ge 1$ and it is an object in $\PPer_{\le 1}(X/Y)$ for 
$i \le 0$. 
On the other hand, since $E\in \tT^{\dag} \cap \fF_{B, f^{\ast}\omega}'$, 
we have 
\begin{align*}
\Hom(\PPer_{\le 1}(X/Y)[i], E)=0, \quad i\ge 0. 
\end{align*}
This contradicts to (\ref{adju:A}), so 
 (\ref{Coh1}) holds. 

By the assumption and (\ref{Coh1}), $\Phi(E)$ fits into 
the distinguished triangle in $D^b \Coh(Y)$
\begin{align*}
U[1] \to \Phi(E) \to F
\end{align*}
such that $U$ is a reflexive sheaf on $Y$
and $F \in \Coh_{\le 1}(Y)$. 
Then the required property for $\mathbb{D}\Phi(E)$ is proved by 
dualizing the above sequence and using (\ref{Coh1}), 
 the property of
the derived dual of reflexive sheaves and
that of one dimensional sheaves. 
The detail is found in the proof of~\cite[Lemma~3.8]{TodBG}. 
\end{proof}
We also show the following: 
\begin{prop}\label{prop:Pfin}
The abelian category $\pP_{B, f^{\ast}\omega}(1)$ is 
of finite length. 
\end{prop}
\begin{proof}
By Proposition~\ref{prop:noether}, $\pP_{B, f^{\ast}\omega}(1)$
is noetherian, so it is enough to show that 
$\pP_{B, f^{\ast}\omega}(1)$ is artinian. 
Suppose that there is an infinite sequence in $\pP_{B, f^{\ast}\omega}(1)$
\begin{align}\label{seq:ar}
E=E_1 \supset E_2 \supset \cdots. 
\end{align}
We take the exact sequence in $\pP_{B, f^{\ast}\omega}(1)$
\begin{align*}
0 \to E_{i} \to E \to F_i \to 0. 
\end{align*}
Since $\ch_1(\ast) (f^{\ast}\omega)^2$ is non-positive
on $\pP_{B, f^{\ast}\omega}(1)$ by (\ref{IZ=0}), we may assume 
that 
$\ch_1(E_i) (f^{\ast}\omega)^2$
is independent of $i$. 
Then $\ch_1(F_i) (f^{\ast}\omega)^2=0$, 
hence $F_i \in \PPer_{0}(X/Y)$ by (\ref{IZ=0}).
In the notation of the proof of Proposition~\ref{prop:noether}, 
we have the surjection in $\PPer_{0}(X/Y)$:
\begin{align*}
\hH_{\bB}^{0}(E) \twoheadrightarrow F_i. 
\end{align*} 
Applying $\Phi$ in Theorem~\ref{thm:tilting}, 
we see that $\Phi(F_i)$ is a zero dimensional sheaf 
whose length is bounded above by the length of 
$\Phi(\hH_{\bB}^{0}(E))$. 
Hence the sequence (\ref{seq:ar}) terminates. 
\end{proof}

\subsection{Proof of Theorem~\ref{thm:intro2}}
Let $M^{\sigma_{B, f^{\ast}\omega}}([\oO_x])$
be the set of isomorphism classes of objects
$E\in \pP_{B, f^{\ast}\omega}(1)$
satisfying $\ch(E)=\ch(\oO_x)$ for $x\in X$. 
Note that by Proposition~\ref{prop:Pfin}, we can 
define $S$-equivalence classes of objects
in $M^{\sigma_{B, f^{\ast}\omega}}([\oO_x])$. 
Also by (\ref{IZ=0}), it follows that 
\begin{align*}
M^{\sigma_{B, f^{\ast}\omega}}([\oO_x])
=\{ E \in \PPer_{0}(X/Y) : \ch(E)=\ch(\oO_x)\}. 
\end{align*}
In particular 
we have $\oO_x \in M^{\sigma_{B, f^{\ast}\omega}}([\oO_x])$
for any $x\in X$. 
We investigate $S$-equivalence classes of 
objects 
$\oO_x$ in the following proposition: 
\begin{prop}\label{prop:key}

(i) For
$x, x' \in X$, 
the objects $\oO_x$ and $\oO_{x'}$
are $S$-equivalent 
in $M^{\sigma_{B, f^{\ast}\omega}}([\oO_x])$
if and only if $f(x)=f(x')$. 

(ii) An object $E\in M^{\sigma_{B, f^{\ast}\omega}}([\oO_x])$
is isomorphic to $\oO_x$ for $x\notin D$
or supported on $D$. In the latter case, 
suppose moreover
that $\Supp(E)$ is connected when $f(D)$ is a curve.
Then $E$ is $S$-equivalent to 
$\oO_x$ for $x\in D$. 
\end{prop}
\begin{proof}

(i) Suppose that $\oO_x$ and $\oO_{x'}$ are $S$-equivalent. 
If $x\notin D$, then $\oO_x$ is a simple 
object in $\PPer_{0}(X/Y)$. Hence $\oO_{x'}$ should 
be isomorphic to $\oO_x$, which implies $x=x'$.
If $x\in D$, then $x' \in D$ by the above argument. 
This implies that $f(x)=f(x')$ when $f(D)$ is a point. 
When $f(D)$ is a curve
and $x, x' \in D$,
we have 
\begin{align}\label{grOx}
\mathrm{gr} (\oO_x) \cong \oO_{L_{f(x)}}(-2)[1] \oplus \oO_{L_{f(x)}}(-1),
\end{align}
by Proposition~\ref{fib:one}.
Therefore we must have $L_{f(x)}=L_{f(x')}$, which 
is equivalent to $f(x)=f(x')$. 

Conversely suppose that $f(x)=f(x')$. 
If $x \notin D$, we have $x=x'$, 
so $\oO_x$ and $\oO_{x'}$ are isomorphic. 
Suppose that $x, x' \in D$. 
When $f(D)$ is a curve, the isomorphism 
(\ref{grOx}) shows that $\oO_x$ and $\oO_{x'}$
are $S$-equivalent. 
When $f(D)$ is a point, 
then 
$\oO_x, \oO_{x'}$ are 
objects in $\PPer_{0}(\widehat{X}/\widehat{Y})$, 
which has only a finite number of simple objects. 
Since $D$ is connected, the argument
 of~\cite[Lemma~3.7]{BaMa2}
shows that
$\oO_x$ and $\oO_{x'}$ are $S$-equivalent. 
This fact can be also checked by using simple objects 
in Proposition~\ref{prop:PDsim} directly. 
For instance if $f$ is type V, we have the resolutions
\begin{align*}
& 0 \to \oO_D(-1) \to S_5 \to \oO_D(-C)^{\oplus 2} \to \oO_x \to 0, \ x 
\in \mathrm{Sing}(D), \\
&0 \to \oO_D(-1) \to S_5 \to \uU 
\to \oO_x \to 0, \ x \notin \mathrm{Sing}(D).
\end{align*}
Therefore for any $x\in D$, $\oO_x$ is $S$-equivalent to 
\begin{align}\label{S-eq2}
\oO_D(-2)[2] \oplus S_5(-1)[1] \oplus \oO_D(-3C)^{\oplus 2}. 
\end{align}

(ii) Let us take an object $E \in \PPer_{0}(X/Y)$. 
Then $E$ is a direct sum of zero dimensional sheaves
supported outside $D$ and an object in $\PPer_{0}(X/Y)$
supported on $D$. 
If we take $b\in \mathbb{Q}$ as in 
Lemma~\ref{lem:key}, then 
any direct summand of $E$ satisfies (\ref{ineq:chB}). 
Therefore if $E$ satisfies 
$\ch(E)=\ch(\oO_x)$ for $x\in X$, 
then $E$ is isomorphic to $\oO_x$ for $x\notin D$
or supported on $D$. 

In the latter case, 
we first discuss the case that $f(D)$ is a curve. 
By Proposition~\ref{fib:one}, 
$E$ is $S$-equivalent to 
\begin{align}\label{Eyy'}
E_{y, y'} \cneq \oO_{L_y}(-2)[1] \oplus \oO_{L_{y'}}(-1)
\end{align}
for some $y, y' \in f(D)$. 
If $\Supp(E)$ is connected, 
then $y=y'$, hence $E$ is $S$-equivalent
to $\oO_x$ for $x \in f^{-1}(y)$ by 
(\ref{grOx}). 

When $f(D)$ is a point, 
we only discuss the type V
case. The other cases are similarly discussed. 
By Proposition~\ref{prop:PDsim} and Proposition~\ref{prop:Persim}, 
$E$ is $S$-equivalent to 
\begin{align*}
\oO_D(-2)[2]^{\oplus a} \oplus S_5(-1)[1]^{\oplus b}
 \oplus \oO_D(-3C)^{\oplus c}
\end{align*}
for some $a, b, c \in \mathbb{Z}_{\ge 0}$. 
By the Chern character computation in the 
proof of Lemma~\ref{lem:key}, 
the condition $\ch(E)=\ch(\oO_x)$ becomes
$a-3b+c=0$, $-a+b=0$ and 
$a/3+b-c/6=1$. 
The solution is $a=b=1$, $c=2$, 
hence $E$ is $S$-equivalent to 
$\oO_x$ for $x\in D$ by (\ref{S-eq2}). 
\end{proof}
\begin{rmk}\label{rmk:Eyy'}
In Proposition~\ref{prop:key} (ii), 
suppose that $f(D)$ is a curve and $\Supp(E)$ is not connected. 
Then the above proof shows that 
$E$ is isomorphic to $E_{y, y'}$ for $y\neq y'$
given 
by (\ref{Eyy'}). 
\end{rmk}
The following is the main 
result in this section: 
\begin{thm}\label{thm:3fold}
We have the following:
\begin{itemize}
\item If $f(D)$ is a curve, then $Y$ is 
one of the irreducible 
components of the coarse moduli space of 
$S$-equivalence classes of objects in 
$M^{\sigma_{B, f^{\ast}\omega}}([\oO_x])$. 
\item If $f(D)$ is a point, then 
$Y$ is the coarse moduli space of 
$S$-equivalence classes of objects in 
$M^{\sigma_{B, f^{\ast}\omega}}([\oO_x])$. 
\end{itemize}
\end{thm}
\begin{proof}
We first discuss the case that $f(D)$ is a point. 
The statement is equivalent to that $Y$ 
corepresents the functor (cf.~\cite[Definition~2.2.1]{Hu})
\begin{align}\label{Mfun}
\mM^{\sigma_{B, f^{\ast}\omega}}([\oO_x]) \colon 
\mathrm{Sch}/\mathbb{C} \to \mathrm{Set}
\end{align}
which assigns a $\mathbb{C}$-scheme $S$ to 
isomorphism classes of objects
\begin{align}\label{famQ}
\qQ \in D^b \Coh(X \times S)
\end{align}
such that for each $s\in S$, 
its derived restriction $\qQ_s$ to 
$X \times \{s\}$ is an object in 
$\PPer_{0}(X/Y)$
with $\ch(\qQ_s)=\ch(\oO_x)$. 
Let us consider the object
\begin{align*}
\dR (f\times \id_{S})_{\ast}\qQ \in \Coh(Y\times S). 
\end{align*}
The above object is a flat family of 
skyscraper sheaves of points in $Y$ over $S$. 
Hence it induces a morphism 
$S \to Y$, giving a natural transformation 
\begin{align*}
F_Y \colon \mM^{\sigma_{B, f^{\ast}\omega}}([\oO_x]) \to \Hom(\ast, Y). 
\end{align*}
Suppose that there is another $\mathbb{C}$-scheme $Z$
and a natural transformation 
\begin{align}\label{nat:Fz}
F_Z \colon \mM^{\sigma_{B, f^{\ast}\omega}}([\oO_x]) \to \Hom(\ast, Z). 
\end{align}
By applying $F_Z$ to the family 
$\{\oO_x\}_{x\in X}$, we obtain a morphism 
\begin{align*}
g \colon X \to Z.
\end{align*}
Note that any two $S$-equivalent objects in 
$\PPer_{0}(X/Y)$ are mapped to the same point 
by $F_Z(\Spec \mathbb{C})$. 
(cf.~\cite[Lemma~4.1.2]{Hu}.)
Therefore by Proposition~\ref{prop:key} (i), 
$g(D)$ must be a point in $Z$. 
Since $f_{\ast}\oO_X=\oO_Y$, 
the morphism $g$ descends to the morphism from $Y$, 
\begin{align*}
h \colon Y \to Z. 
\end{align*}
We need to show that
\begin{align*}
h_{\ast} \circ F_{Y} = F_{Z}.
\end{align*}
The above relationship 
follows from Proposition~\ref{prop:key} (ii). 
Hence $Y$ corepresents the functor (\ref{Mfun}). 

In the case that $f(D)$ is a curve, 
we consider the scheme
\begin{align*}
\widetilde{Y} \cneq Y \cup (f(D) \times f(D)),
\end{align*}
where $Y$ and $f(D) \times f(D)$ are glued 
along $f(D) \subset Y$ and the diagonal 
$f(D) \subset f(D) \times f(D)$. 
The scheme structure of $\widetilde{Y}$ along 
the intersection is given by the fiber product
\begin{align*}
\oO_{\widetilde{Y}}=\oO_{Y} \times_{\oO_{f(D)}} \oO_{f(D) \times f(D)}. 
\end{align*}
Note that for $p \in f(D)$, the ring 
$\widehat{\oO}_{\widetilde{Y}, p}$ is written as 
\begin{align*}
\widehat{\oO}_{\widetilde{Y}, p}
\cong \mathbb{C} \db [x, y, z, w \db] /(xz, yz). 
\end{align*}
We show that $\widetilde{Y}$ corepresents the functor 
(\ref{Mfun}).
Let $S$, $\qQ$ be as before, 
and $S_1, \cdots, S_N$ be the irreducible components of $S$. 
If $\Supp \qQ_s$ is connected for general (hence for any)
point $s\in S_i$, then 
we have the morphism $S_i \to Y$ 
induced by the flat family of skyscraper sheaves on $Y$ over $S_i$, 
\begin{align}\label{RQ1}
\dR (f\times \id_{S_i})_{\ast}(\qQ|_{X\times S_i})
\in \Coh(Y\times S_i). 
\end{align}
If $\Supp \qQ_s$ is not connected for 
general $s\in S_i$, we consider another object
\begin{align}\label{RQ2}
\dR (f\times \id_{S_i})_{\ast}(\qQ|_{X\times S_i}(-D \times S_i)) \in 
\Coh(Y\times S_i).
\end{align}
The objects (\ref{RQ1}), (\ref{RQ2}) 
are flat families of skyscraper
sheaves of points in $f(D)$ over $S_i$, 
hence induce a morphism 
\begin{align*}
S_i \to f(D) \times f(D).
\end{align*}
The above morphisms on irreducible components 
glue on the intersections, so
induce a (unique)
 morphism 
$S \to \widetilde{Y}$. 
In this way, we obtain a natural transformation
\begin{align*}
F_{\widetilde{Y}} 
\colon \mM^{\sigma_{B, f^{\ast}\omega}}([\oO_x]) 
\to \Hom(\ast, \widetilde{Y}). 
\end{align*}
Suppose that there is another natural transformation $F_Z$
as in (\ref{nat:Fz}). 
Similarly to the case that $f(D)$ is point, 
$F_Z \left(\{\oO_x\}_{x\in X}\right)$
induces the
morphism 
$Y \to Z$. Also 
applying $F_Z$ to the family 
\begin{align*}
\{E_{y, y'}\}_{(y, y') \in f(D) \times f(D)}
\end{align*}
given by (\ref{Eyy'}),  
we obtain a morphism
\begin{align*}
f(D) \times f(D) \to Z. 
\end{align*}
These morphisms glue along 
the intersection, giving a morphism 
$\widetilde{h} \colon \widetilde{Y} \to Z$. 
Similarly to the case that $f(D)$ is a point, the 
relationship 
\begin{align*}
\widetilde{h}_{\ast} \circ F_{\widetilde{Y}} = F_{Z}
\end{align*}
follows from Proposition~\ref{prop:key} (ii) and Remark~\ref{rmk:Eyy'}. 
Therefore $\widetilde{Y}$ corepresents the functor (\ref{Mfun}), 
and $Y$ is the desired irreducible component. 
\end{proof}
\begin{rmk}\label{rmk:glue}
Since 
Proposition~\ref{prop:key} (ii)
is not true when $f(D)$ is a curve and 
$\Supp(E)$ is not connected, the 
statement of Theorem~\ref{thm:3fold}
is weaker in this case. Indeed 
the objects $E_{y, y'}$ in (\ref{Eyy'})
provide another component $f(D) \times f(D)$. 
\end{rmk}

\begin{rmk}
Let us consider a 
conjectural Bridgeland stability condition in 
(\ref{BMT}), 
and a category 
$\pP_{B, f^{\ast}\omega -\varepsilon D}(1)$ defined 
similarly to Definition~\ref{def:P(1)}. 
Then $X$ is shown to be the fine 
moduli space of objects $E\in \pP_{B, f^{\ast}\omega-\varepsilon D}(1)$
with $\ch(E)=\ch(\oO_x)$. 
If the relation (\ref{BMT:lim}) holds, 
then we can consider a one parameter family 
\begin{align*}
\sigma_t =(Z_{B, f^{\ast}\omega +\varepsilon tD}, 
 \pP_t), \quad t\in (-1, 1),
\end{align*}
similarly to the surface case. 
The behavior of moduli spaces of 
$E \in \pP_t(1)$ with $\ch(E)=\ch(\oO_x)$
under change of $t\in (-1, 0]$ 
is similar to the surface 
case in Theorem~\ref{thm:main1}, by Theorem~\ref{thm:3fold}.
On the other hand, 
it would be interesting to
study such a moduli space for $t>0$. 
The moduli space itself may be
studied by investigating 
the space of stability conditions on $D^b (\PPer_{0}(X/Y))$, 
which is much easier than $\Stab(X)$. 
\end{rmk}

Todai Institute for Advanced Studies (TODIAS), 

Kavri Institute for the Physics and 
Mathematics of the Universe, 

University of Tokyo, 
5-1-5 Kashiwanoha, Kashiwa, 277-8583, Japan.

\textit{E-mail address}: yukinobu.toda@ipmu.jp

\end{document}